\documentclass[11pt]{amsart}
\usepackage{amscd}
\usepackage{amsmath}
\usepackage{amsxtra}
\usepackage{amsfonts}
\usepackage{amssymb}

 \def\lcm{\operatorname*{lcm}}
\newtheorem{theorem}{Theorem}[section]
\newtheorem{corollary}[theorem]{Corollary}
\newtheorem{lemma}[theorem]{Lemma}
\newtheorem{proposition}[theorem]{Proposition}

\newtheorem{conjecture}[theorem]{Conjecture}
\theoremstyle{definition}
\newtheorem{definition}[theorem]{Definition}
\newtheorem{remark}[theorem]{Remark}

\newtheorem{question}[theorem]{Question}
\newtheorem{example}[theorem]{Example}
\theoremstyle{remark}

\renewcommand{\theclaim}{\textup{\theclaim}}

\numberwithin{equation}{section}

\def\openone

{\mathchoice

{\hbox{\upshape \small1\kern-3.3pt\normalsize1}}

{\hbox{\upshape \small1\kern-3.3pt\normalsize1}}

{\hbox{\upshape \tiny1\kern-2.3pt\SMALL1}}

{\hbox{\upshape \Tiny1\kern-2pt\tiny1}}}

\makeatletter

\newbox\ipbox

\newcommand{\diracb}[1]{\left\langle #1\mathrel{\mathchoice

{\setbox\ipbox=\hbox{$\displaystyle \left\langle\mathstrut
#1\right.$}

\vrule height\ht\ipbox width0.25pt depth\dp\ipbox}

{\setbox\ipbox=\hbox{$\textstyle \left\langle\mathstrut
#1\right.$}

\vrule height\ht\ipbox width0.25pt depth\dp\ipbox}

{\setbox\ipbox=\hbox{$\scriptstyle \left\langle\mathstrut
#1\right.$}

\vrule height\ht\ipbox width0.25pt depth\dp\ipbox}

{\setbox\ipbox=\hbox{$\scriptscriptstyle \left\langle\mathstrut
#1\right.$}

\vrule height\ht\ipbox width0.25pt depth\dp\ipbox}

}\right. }

\newcommand{\dirack}[1]{\left. \mathrel{\mathchoice

{\setbox\ipbox=\hbox{$\displaystyle \left.\mathstrut
#1\right\rangle$}

\vrule height\ht\ipbox width0.25pt depth\dp\ipbox}

{\setbox\ipbox=\hbox{$\textstyle \left.\mathstrut
#1\right\rangle$}

\vrule height\ht\ipbox width0.25pt depth\dp\ipbox}

{\setbox\ipbox=\hbox{$\scriptstyle \left.\mathstrut
#1\right\rangle$}

\vrule height\ht\ipbox width0.25pt depth\dp\ipbox}

{\setbox\ipbox=\hbox{$\scriptscriptstyle \left.\mathstrut
#1\right\rangle$}

\vrule height\ht\ipbox width0.25pt depth\dp\ipbox}

} #1\right\rangle}

\newcommand{\bz}{\mathbb{Z}}

\newcommand{\br}{\mathbb{R}}

\newcommand{\bn}{\mathbb{N}}

\newcommand{\beq}{\begin{equation}}
\newcommand{\eeq}{\end{equation}}

\def\blfootnote{\xdef\@thefnmark{}\@footnotetext}

\renewcommand{\mod}{\operatorname{mod}}

\input xy
\xyoption{all}
\usepackage{amssymb}

\def\N{\mathbb{N}}

\def\-{^{-1}}

\def\Z{\mathbb{Z}}

\begin{document}

\title[Scaling of spectra]{Number theoretic considerations related to the scaling of spectra of Cantor-type measures}
\author{Dorin Ervin Dutkay}

\address{[Dorin Ervin Dutkay] University of Central Florida\\
	Department of Mathematics\\
	4000 Central Florida Blvd.\\
	P.O. Box 161364\\
	Orlando, FL 32816-1364\\
U.S.A.\\} \email{Dorin.Dutkay@ucf.edu}

\author{Isabelle Kraus}

\address{[Isabelle Kraus] University of Central Florida\\
	Department of Mathematics\\
	4393 Andromeda Loop N.\\
	Orlando, FL 32816-1364\\
U.S.A.\\} \email{izzy.kraus@knights.ucf.edu}
\subjclass[2010]{11A07,11A51,42C30}    
\keywords{Cantor set, Fourier basis, prime decomposition, spectral measure, base $n$ decomposition}
\begin{abstract}
We investigate some relations between number theory and spectral measures related to the harmonic analysis of a Cantor set. Specifically, we explore ways to determine when an odd natural number $m$ generates a complete or incomplete Fourier basis for a Cantor-type measure with scale $g$.
 \end{abstract}
\maketitle \tableofcontents

\section{Introduction}
In \cite{JP98}, Jorgensen and Pedersen constructed the first example of a singular fractal measure on a Cantor set, which has an orthonormal Fourier series. This Cantor set is obtained from the interval $[0,1]$, dividing it into four equal intervals and keeping the first and the third, $[0,1/4]$ and $[1/2,3/4]$, and repeating the procedure infinitely many times. The measure $\mu_4$ on this Cantor set associates measure 1 to $[0,1]$, measure $\frac12$ to $[0,1/4]$ and $[2/4,3/4]$, measure $\frac14$ to the four intervals in the next step of the construction and so on. It is the Hausdorff measure of dimension $\frac12$ on this Cantor set, and it is also the invariant measure of the iterated function system $\tau_0(x)=x/4$, $\tau_2(x)=(x+2)/4$ (see \cite{Hut81} or \cite{JP98} for details). 

Jorgensen and Pedersen proved the surprising result that the Hilbert space $L^2(\mu_4)$ has an orthonormal basis formed with exponential functions, i.e., a Fourier basis, $E(\Gamma_0):=\{e^{2\pi i\lambda x} : \lambda\in\Gamma_0\}$ where 
\begin{equation}
\Gamma_0:=\left\{\sum_{k=0}^n4^kl_k : l_k\in\{0,1\},n\in\bn\right\}.
\label{eqi1.3}
\end{equation}

A set $\Lambda$ in $\br$ is called a {\it spectrum} for a Borel probability measure $\mu$ on $\br$ if the corresponding exponential functions $\{e^{2\pi i\lambda x} : \lambda\in\Lambda\}$ form an orthonormal basis for $L^2(\mu)$. 

Jorgensen and Pedersen's example opened up a new area of research and many other examples of singular measures which admit orthonormal Fourier series have been constructed since, see e.g., \cite{MR1785282,MR1929508,DJ06,DJ07d,MR3055992,MR2297038}. 

In \cite{DJ12}, it was proved also that the set $5^k\Gamma_0$ is a spectrum for the measure $\mu_4$, for any $k\in\bn$! This means that the operator on $L^2(\mu_4)$ which maps $e^{2\pi i\lambda x}$ into $e^{2\pi i 5^k\lambda x}$ is actually unitary, for all $k$, which means that there are some hidden symmetries, a certain scaling by 5 in the geometry of this Cantor set. These operators were further investigated in \cite{MR2966145,MR3202868,MR3310953}.

Later, Dutkay and Haussermann \cite{DuHa16} studied for what digits $\{0,m\}$, with $m$ odd, the set 
$$\Gamma(m):=m\Gamma_0=\left\{\sum_{k=0}^n 4^kl_k : l_k\in\{0,m\}, n\in\bn\right\}$$
is a spectrum for $L^2(\mu_4)$. Among other things, they proved that, for any prime number $p>3$, the set $p^k\Gamma_0$ is a spectrum for $\mu_4$, and there are some interesting number theoretic considerations that are required to solve this problem. 

Now, we will generalize these results. 

Consider the iterated function system generated by a scale $g$, with $g$ even, and the digits $B=\{0,\frac g2\}$,
$$\tau_0(x)=\frac{x}{g},\quad \tau_{g/2}(x)=\frac{x+g/2}{g}.$$
 Let $\mu$ be the invariant measure for this iterated function system. This is the unique Borel probability measure on $\br$ which satisfies the invariance equation 
$$\mu(E)=\frac12\left(\mu(\tau_0^{-1}(E))+\mu(\tau_{g/2}^{-1}(E))\right),\mbox{ for all Borel sets $E$},$$ 
(see \cite{Hut81}).

We want to find the answer to the following question:
\begin{question}\label{q1}
 For what digits $\{0,m\}$ is the set 
\begin{equation}
\Gamma(m):=m\Gamma(1)=\left\{\sum_{k=0}^n g^kl_k : l_k\in\{0,m\}, n\in\bn\right\}
\label{eqgammam}
\end{equation}
a spectrum for $L^2(\mu)$?
\end{question}

As in \cite{DJ06}, we look for Hadamard triples of the form $(R,B,L)$ with $L=\{0,m\}$. That means that the matrix
$$\frac{1}{\sqrt 2}\left(e^{2\pi  i\frac{bl}R}\right)_{b\in B,l\in L}$$
is unitary, so $e^{2\pi i\frac{(g/2)\cdot m}{g}}=-1$. This means that $m$ is odd.

 It was shown \cite{DJ06} that the numbers \textit{m} that give spectra can be characterized in terms of {\it extreme cycles}, i.e., we want to find the even integers $g$ for which there exist $l_0,\dots, l_{r-1}\in\{0,m\}$, not all equal to 0, such that
\begin{equation}
x_1=\frac{x_0+l_0}g,\ x_2=\frac{x_1+l_1}g,\ \dots,x_{r-1}=\frac{x_{r-2}+l_{r-2}}g,\ x_0=\frac{x_{r-1}+l_{r-1}}g,
\label{eqec}
\end{equation}

and 
\begin{equation}
\left|\frac{1+e^{2\pi i \frac g2 x_k}}{2}\right|=1,\quad(k\in\{0,\dots,r-1\})
\label{eq1.1}
\end{equation}

where the finite set $\{x_0, x_1, \dots, x_{r-1}\}$ is the extreme cycle for $\{0,m\}$, and $x_i$ are the extreme cycle points. If such an extreme cycle exists then the set of exponential functions corresponding to $\Gamma(m)$ is incomplete but orthonormal, if no such extreme cycle exists then the set of exponential functions corresponding to $\Gamma(m)$ is an orthonormal basis, i.e., $\Gamma(m)$ is a spectrum. 

We note that points $x_0,\dots,x_{r-1}$ have to be integers. Indeed, equation \eqref{eq1.1} implies that $x_i=\frac{k_i}{g/2}$, for some $k\in\bz$. Assume that $k_0$ is not divisible by $g/2$. We have 
$$\frac{\frac{k_0}{g/2}+l_0}{g}=\frac{x_0+l_0}{g}=x_1=\frac{k_1}{g/2}$$
for some integer $l$. 

Then $\frac{k_0}{g/2}=2k_1-l_0$ so $k$ has to be divisible by $g/2$, contradiction. Thus, $x_0$ is in $\bz$ so all the points in the extreme cycle have to be integers. 

Hence, Question \ref{q1} becomes a purely number theoretical question:

\begin{question}\label{q2}
Given an even number $g\geq 4$, for what odd numbers $m\geq 1$, are there non-trivial extreme cycles, i.e., finite sets $C=\{x_0,\dots, x_{r-1}\}$ of integers and digits $l_0,\dots,l_{r-1}\in\{0,m\}$ such that 
$$x_1=\frac{x_0+l_0}g,\ x_2=\frac{x_1+l_1}g,\ \dots,x_{r-1}=\frac{x_{r-2}+l_{r-2}}g,\ x_0=\frac{x_{r-1}+l_{r-1}}g?$$
\end{question}

The extreme cycle $\{0\}$ corresponding to the digit $0$, is called the {\it trivial} extreme cycle.

\begin{definition}
We say that $m$ is {\it complete} if the only extreme cycle for the digit set $\{0,m\}$ is the trivial one $\{0\}$. Otherwise, $m$ is {\it incomplete}. 
In the paper, when we refer to an extreme cycle, we will assume it is not trivial. 
\end{definition}
As we mentioned above , if $m$ is complete then the set $\Gamma(m)$ in \eqref{eqgammam} is a complete orthonormal basis, and if $m$ is incomplete then $\Gamma(m)$ is an incomplete orthonormal set in $L^2(\mu)$ (see \cite{DJ06}). 

According to Lemma \ref{lem2.5}, any odd multiple of an incomplete number is also incomplete. This justifies the next definition. 

\begin{definition}
We say that an odd number $m$ is $primitive$ if $m$ is incomplete and, for all proper divisors $d$ of $m$, $d$ is complete. In other words, there exist non-trivial extreme cycles for the digits $\{0,m\}$ and there are no non-trivial extreme cycles for the digits $\{0,d\}$ for any proper divisor $d$ of $m$. We say that a primitive number $m$ is non-trivial if $m\neq g-1$. 
\end{definition}

In Theorems \ref{th2.12}, \ref{th2.13} and \ref{th2.18} we present some large classes of numbers which are complete (such as prime powers, in most cases). On the other hand, in Theorem \ref{th2.17}, we show that there are infinitely many primitive numbers. 

Finding explicit formulas for primitive numbers (and thus for incomplete or complete numbers) seems to be a difficult task, even for particular choices of number $g$. 

In Section \ref{sec2.4} we study the connection between primitive numbers and their order relative to $g$ (see Definition \ref{def2.8}). The key technical tool is Proposition \ref{prop2.23}. In Theorems \ref{th4.1} and \ref{th4.8.1} we study primitive numbers of small order. In Theorem \ref{th4.9} we present the form of a primitive number in terms of its order and the digits corresponding to its extreme cycle. Theorem \ref{th2.23} presents a way to locate primitive numbers and in Theorem \ref{th2.34} we give an explicit example of a non-trivial primitive number. 

Since, from Theorem \ref{th2.18}, we see that, in most cases, prime powers are complete, in Section \ref{sec2.6}, we study classes of composite numbers which are complete. The results are based on some important technical lemmas : Lemma \ref{lem2.4}, \ref{lem2.27.5} and \ref{lem2.29}. In Section \ref{sec2.6}, we use these lemmas in various ways to obtain  new composite numbers which are complete, from simpler complete composite numbers.  
\section{Main results}

For the rest of the paper $g$ will be an even integer $g\geq 4$ and $m$ will be an odd integer $m\geq1$.

\subsection{Some preliminary lemmas}

\begin{lemma}\label{lemi1.4}
If $x_0\in\bz$ is an extreme cycle point with digits $l_0,\dots,l_{r-1}$ as in \eqref{eqec}, then $x_0$ has a periodic base $g\in\bn$ expansion, 
\begin{equation}
x_0=\frac{l_{r-1}}g+\frac{l_{r-2}}{g^2}+\dots+\frac{l_0}{g^r}+\frac{l_{r-1}}{g^{r+1}}+\dots+\frac{l_0}{g^{2r}}+\dots,
\label{eqi1.4.1}
\end{equation}
and $0< x_0\leq\frac {m}{g-1}$. 
We write this as $x_0=.\underline{l_{r-1}l_{r-2}\dots l_1l_0}$, the underline indicates the infinite repetition of the digits $l_{r-1}\dots l_0$ in the base $g$ expansion  of $x_0$. 

Hence
$$x_0=\frac{g^{r-1}l_{r-1}+g^{r-2}l_{r-2}+\dots+gl_1+l_0}{g^r-1}.$$
Moreover, $$\{x_0 : x_0 \text{ is an extreme cycle point } \} = X_L \cap \mathbb{Z},$$ 
where $X_L$ is the attractor of the iterated function system
$$\sigma_0(x) = \frac{x}{g}, \,\,\sigma_m(x) = \frac{x+m}{g},$$
so
$$X_L = \cup_{l \in \{0,m\}} \sigma_l(X_L),$$
\begin{equation}
X_L = \Bigg\{\sum\limits_{n=1}^{\infty} \frac{l_n}{g^n} : l_n \in \{0,m\} \text{ for all n} \in \mathbb{N} \Bigg\}.
\label{eqi1.4.2}
\end{equation} 
\end{lemma}

\begin{proof}
Recall that a finite set $\{x_0, x_1, \dots , x_{r-1}\}$ is an extreme cycle for digits $\{0,m\}$ if there exist $l_0, \dots, l_{r-1}\in \{0,m\}$ such that
$$x_1=\frac{x_0+l_0}g,\ x_2=\frac{x_1+l_1}g,\ \dots , x_{r-1}=\frac{x_{r-2}+l_{r-2}}g,\ x_0=\frac{x_{r-1}+l_{r-1}}g.$$
$$x_0=\frac{x_{r-1}}g+\frac{l_{r-1}}{g}=\frac{x_{r-2}}{g^2}+\frac{l_{r-2}}{g^2}+\frac{l_{r-1}}{g}=\dots=\frac{x_0}{g^r}+\frac{l_0}{g^r}+\frac{l_1}{g^{r-1}}+\dots+\frac{l_{r-1}}{g}.$$ 
Iterating this equality to infinity we obtain the base $g$ decomposition of $x_0$. Also 
$$0< x_0\leq \sum_{k=1}^\infty \frac{m}{g^k}=\frac{m}{g-1}.$$

From above, we know that $x_0 \in \mathbb{Z}$. Therefore, $x_0$ is contained in $X_L \cap \mathbb{Z}$. Conversely, if $x_0 \in X_L \cap \mathbb{Z}$ then, if $x_0 \in \sigma_0(X_L)$, we have that there exists $x_{-1} \in X_L$ such that $x_0 = \frac{x_{-1}}{g},$ and we get that $x_{-1} = gx_0 \in \mathbb{Z} \cap X_L$. If $x_0 \in \sigma_m(X_L)$ then there exists $x_{-1} \in X_L$ such that $x_0 = \frac{x_{-1}+m}{g}$. Then $x_{-1} = gx_0 - m \equiv x_0 (\mod m)$. By induction, we obtain $x_{-1}, x_{-2} \ldots$ in $X_L\cap \bz$ and digits $d_0, d_1, \ldots$ in $\{0,m\}$ such that $x_{-i} = \frac{x_{-i-1}+d_i}{g}$. Since the set $X_L\cap\bz$ is finite it follows that there exists $k$ and $p$, $k<p$, such that $x_{-k}=x_{-p}$. That means that $\{x_{-k},x_{-k-1},\dots,x_{-p}\}$ form a cycle. We will show that actually we can start the cycle with $x_0$. 

We have that $\frac{x_{-k}+d_{k-1}}{g}=x_{-k+1}\in\bz$. Also $\frac{x_{-k}+d_{p-1}}{g}=\frac{x_{-p}+d_{p-1}}{g}=x_{-p+1}\in\bz$. This means that $d_{k-1}-d_{p-1}$ is divisible by $g$, and since the only digits we use are $0$ and $m$ and $m$ (odd) is not divisible by $g$ (even), it follows that $d_{k-1}=d_{p-1}$ and therefore $x_{-k+1}=x_{-p+1}$. By induction, we get that $x_0$ must be in the same cycle.

\end{proof}

\begin{lemma}\label{lem2.3}
Assume $m$ is odd and $x_j$ is an extreme cycle point for the digit set $\{0,m \}$. Then $x_j \equiv 0 (\mod g)$ or $x_j \equiv -m (\mod g)$.
\end{lemma}

\begin{proof}
We have $x_{j+1} = \frac{x_j + l_j}{g}$ where $l_j \in \{0,m \}$. 
Then $gx_{j+1} = {x_j + l_j}.$ 
If $l_{j}$ is 0, we get that $g x_{j+1} = x_j.$ 
Otherwise, $g x_{j+1} = x_{j} + m.$
Considering these modulo $g$, we have $0 \equiv x_j +m (\mod g)$
or $0 \equiv x_j (\mod g)$.
Thus
$x_{j} \equiv -m (\mod g)$
or
$x_{j} \equiv 0 (\mod g)$.
\end{proof}

\begin{lemma}\label{lem2.4}
Let $m$ be an odd number not divisible by $g-1$ and $x_t$ be the largest extreme cycle point in the non-trivial extreme cycle $X$ for the digit set $\{0,m \}$. Then $x_t$ is divisible by $g$.
\end{lemma}

\begin{proof}
Assume by contradiction that $x_t$ is not divisible by $g$. Then, we know that the next cycle point is 
$$x_{t+1} = \frac{x_t + m}{g}. $$

Since $x_t$ is the largest cycle point in this cycle, we have that $\frac{x_t+m}{g}\leq x_{t}$. If $\frac{x_t+m}{g}=x_t$ then $x_t=\frac{m}{g-1}$ so $m$ is divisible by $g-1$, contradiction. Otherwise, we get that $x_t>\frac{m}{g-1}$, which contradicts Lemma \ref{lemi1.4}.
\end{proof}

\begin{lemma}\label{lem2.4.1}
If $m=g-1$, then it has the extreme cycle $\{1\}$ with digits $\{m\}$.
\end{lemma}

\begin{proof}
Letting $m = g-1$ and $x_0 = 1$, we get that $x_1 = \frac{1 + g-1}{g}$. So, $x_1 = 1$, and since $x_1 = x_0$, $\{1\}$ is indeed an extreme cycle for the digit $g-1.$
\end{proof}

\begin{lemma}\label{lem2.5}
If $m$ is incomplete, then any odd multiple of $m$ is also incomplete. 
\end{lemma}

\begin{proof}
The number $m$ is complete if and only if the only extreme cycle for the digit set $\{0,m\}$ is the trivial one $\{0\}$. Suppose that $m$ is incomplete, so $m$ has the non-trivial extreme cycle $\{x_0, x_1, ... , x_{r-1}\}$, where $$x_1=\frac{x_0+l_0}g,\ x_2=\frac{x_1+l_1}g,\ \dots,x_{r-1}=\frac{x_{r-2}+l_{r-2}}g,\ x_0=\frac{x_{r-1}+l_{r-1}}g.$$
Consider the extreme cycles of $\{0, km\}$, where $k$ is an odd number. Multiplying the previous expression by $k$, we get that, 
$$kx_1=\frac{kx_0+kl_0}g,\ kx_2=\frac{kx_1+kl_1}g,\ \dots,kx_{r-1}=\frac{kx_{r-2}+kl_{r-2}}g,\ kx_0=\frac{kx_{r-1}+kl_{r-1}}g.$$
Thus, we get an extreme cycle for the digit $km$. 

Hence, odd multiples of $m$ are incomplete whenever $m$ is incomplete. 
\end{proof}

\begin{lemma}\label{lem2.6}
All of the odd numbers between $1$ and $g-2$ are complete. 
\end{lemma}

\begin{proof}
Let $m$ be an odd number, $1\leq m\leq g-2$. Suppose $m$ is incomplete. Then, by Lemma \ref{lemi1.4}, the set $[0,\frac{m}{g-1}]\cap\bz\supset X_L\cap \bz$ contains a cycle point so it is non-empty. But $\frac{m}{g-1}<1$, and this is a contradiction. 
\end{proof}

\begin{lemma}\label{lem2.7}
Let $x_0$ be a cycle point, i.e., it has the form in \eqref{eqi1.4.1}. Suppose $x_0$ is an integer. Then $x_0$ is an extreme cycle point, i.e., all the other points in the cycle are integers. 
\end{lemma}

\begin{proof}
Suppose that $\{x_0, x_1, ..., x_n\}$ is a cycle for $\{0,m\}$, $l_n$ $\in$ $\{0,m\}$ and that $x_0$ $\in \bz$. Since $x_0$ is a cycle point, we know that 
$x_0 = \frac{x_n+l_n}{g}$. Then $x_n = x_0g-l_n$, so $x_n$ $\in \bz$. By induction, all points in the cycle are integers.
\end{proof}

\begin{definition}\label{def2.8}
Let $m$ be an odd natural number. We will denote by $\Z_m$ the finite ring of integers modulo $m$. We denote by $U(\Z_m)$ the multiplicative group of elements in $\Z_m$ that have a multiplicative inverse, i.e., the elements in $\bz_m$ which are relatively prime with $m$. For $a\in U(\bz_m)$, we denote by $o_a(m)$ the order of the element $a$ in the group $U(\bz_m)$. We also say that $m$ {\it has order }$o_a(m)$ (with respect to $a$). We denote by $G_{m,g}$ (or $G_m$) the group generated by $g$ in $U(\Z_m)$, that is $G_{m,g} = \{g^j(\mod m) : j = 0, 1, \dots\}.$ 
\end{definition} 

\begin{proposition}\label{pr2.14}
Assume $m>g-1$ is odd. If a coset $C$ of $G_{m,g}$ in $U(\Z_m)$ has the property that for all $x_j \in C$, $x_j < \frac{2m}{g}$, then $C$ is an extreme cycle for the digit set $\{0,m\}$. 
\end{proposition}

\begin{proof}
Let $C$ be such a coset. Label the elements in $C$ such that $x_j \equiv gx_{j+1}(\mod m)$, and if $a$ is the number of elements in $G_{m,g}$, then $x_{a-1} \equiv gx_{0}(\mod m)$. Then, since $0 < x_{j+1} < \frac{2m}{g}$, we have that $0 < gx_{j+1} < 2m$. Now, since $x_j \equiv gx_{j+1}(\mod m)$, we have that $x_j = gx_{j+1} + km$, where $k \in \Z$. Consider the following possibilities for the value of $k$.
 
If $k>0$, then $ x_j=gx_{j+1}+km > m> \frac{2m}{g}$, a contradiction. 

If $k \leq -2$, then $x_j \leq gx_{j+1} - 2m \implies x_j \leq 0$, a contradiction. 

So, $k \in \{0, -1\}$, and it follows that, $x_j = gx_{j+1} - km$ for $k \in \{0,1\}$, and similarly for $x_0$ and $x_{a-1}$. Rearranging, we find that $$\frac{x_j + l_j}{g} = x_{j+1}$$ for $l_j \in \{0,m\}$, and similarly for $x_0$ and $x_{a-1}$. Since $C$ contains only integers, $C$ is an extreme cycle.
\end{proof}

\subsection{Some complete numbers}

\begin{theorem}\label{th2.12}
Let $m>g-1$ be an odd  number not divisible by $g-1$. If any of the numbers $-1(\mod m), -2(\mod m),  \dots , -g+2(\mod m)$ or $2(\mod m), 3(\mod m),  \dots , g-1(\mod m)$ is in $G_{m,g}$, then $m$ is complete. 
\end{theorem}

\begin{proof}
Assume by contradiction that $m$ is incomplete. Then there is a non-trivial extreme cycle $C = \{x_0, \dots, x_{r-1}\}$ for the digit set $\{0,m\}$. From the relation between the cycle points, $$x_{j+1} = \frac{x_j+l_j}{g},$$ where $l_j \in \{0,m\}$, we have that $gx_{j+1} \equiv x_j(\mod m)$. Thus, $$g^{r-k}x_0 \equiv x_k(\mod m), \text { with } k \in \{0, \dots, r\}, x_r:=x_0.$$ So, for all $k \in \mathbb{N}$, the number $g^kx_0$ is congruent modulo $m$ with an element of the extreme cycle $C$. If, as in the hypothesis, there is a number $c \in \{-1, -2, \dots, -g+2\}$ in $G_{m,g}$ such that the number $cx_0$ is congruent modulo $m$ with an element in $C$, and since $x_0$ is arbitrary in the cycle, we get that $cx_j$ is congruent to an element in $C$ for any $j$. 

In the following arguments, we use the fact that since $m$ is not divisible by $g-1$, the condition on cycle points $0<x_j \leq \frac{m}{g-1}$ implies $0<x_j<\frac{m}{g-1}$. 

If $c \in \{-1, -2, \dots, -g+2\}$, then, since $0<x_0<\frac{m}{g-1}$, we have $0>cx_0>-m$ so $$cx_0(\mod m) = m+cx_0 > m+c\frac{m}{g-1} = \frac{m(g+c-1)}{g-1} > \frac{m}{g-1},$$ a contradiction with the fact that $cx_0(\mod m)$ is a cycle point. 

For the second set $ \{2, 3, \dots, g-1\}$, by a similar argument, we have that for some $c$ in this set, $cx_j(\mod m) \in C$ for all $j$. Let $x_N$ be the largest element of the extreme cycle. Since $0 < x_N < \frac{m}{g-1}$, $0<cx_N<m$ so $cx_N(\mod m) = cx_N>x_N$, a contradiction to the maximality of $x_N$. 
\end{proof}

\begin{theorem}\label{th2.13}
Let $m>g(g-1)$ be an odd number not divisible by $g-1$. If any of the numbers $g+1(\mod m), g+2(\mod m), \dots ,$ or $g(g-1)(\mod m)$ is in $G_{m,g}$, then $m$ is complete. 
\end{theorem}

\begin{proof}
Assume by contradiction that $m$ is incomplete. Then there is a non-trivial extreme cycle $C = \{x_0, \dots, x_{r-1}\}$ for the digit set $\{0,m\}$. As in the proof of Theorem \ref{th2.12}, for all $k \in \N$, the number $g^{k}x_0$ is congruent modulo $m$ with an element of the extreme cycle $C$. But then, the hypothesis implies that there is a number $c \in \{g+1,g+2, \dots , g(g-1)\}$ in $G_{m,g}$ such that the number $cx_0$ is congruent modulo $m$ with an element in $C$, and since $x_0$ is arbitrary in the cycle, we get that $cx_j$ is congruent to an element in $C$ for any $j$. 

In the following arguments we use the fact that since $m$ is not divisible by $g-1$, the condition on the cycle points $0< x_j \leq \frac{m}{g-1}$ implies $0 < x_j < \frac{m}{g-1}$. Let $x_t$ be the largest element in the extreme cycle. We have $$0 < x_t < \frac{m}{g-1}.$$
By Lemma \ref{lem2.4}, $x_t$ is divisible by $g$. Therefore, dividing by $g$, we get the next element in the extreme cycle, called $x_N$, and we have $$x_N < \frac{m}{g(g-1)}.$$
For $c \in \{g+1, g+2, \dots , g(g-1)\}$,  $x_t=gx_N < cx_N < m$, so $cx_N(\mod m) = cx_N$ is a point in $C$ bigger than $x_t$, a contradiction to the maximality of $x_t$. 
\end{proof}

\begin{corollary}

For $n \geq 1$, the numbers $g^n+1, g^n+3, \dots, g^n+(g-1)$ are complete. 
For $n \geq 2$, the numbers $g^n-3, g^n-5, \dots, g^n-(g-1)$ are complete.
For $n \geq 3$, the numbers $g^n-(g+1), g^n-(g+3), \dots, g^n-(g(g-1)-1)$ are complete.  
\end{corollary}

\begin{proof}
Let $n \geq 1$ and $m = g^n+1$. Then $g\equiv 1(\mod (g-1))$ so $g^n\equiv 1(\mod (g-1))$, so $g^n+1\equiv 2(\mod (g-1))$ so $m$ is not divisible by $g-1$. Then $g^n \equiv -1(\mod m)$. Since $g^n \in G_{m,g}$, by Theorem \ref{th2.12}, $-1 \in G_{m,g}$, so $m$ is complete. Similarly for $g^n+3, g^n+5, \dots, g^n+(g-1)$. 

Let $n \geq 2$ and $m = g^n-3$. Then $g^n-3\equiv -2\equiv g-3(\mod (g-1))$ so $m$ is not divisible by $g-1$. Also $g^n \equiv 3(\mod m)$. Since $g^n \in G_{m,g}$, by Theorem \ref{th2.12}, $3 \in G_{m,g}$, so $m$ is complete. Similarly for $g^n-5, g^n-7, \dots, g^n-(g-1)$.

Let $n \geq 3$ and $m = g^n-(g+1)$. Then $g^n-(g+1)\equiv -g(\mod(g-1))$ so $m$ is not divisible by $g-1$. Also $g^n \equiv (g+1)(\mod m)$. Since $g^n \in G_{m,g}$, by Theorem \ref{th2.13} $(g+1) \in G_{m,g}$, so $m$ is complete. Similarly for $g^n-(g+3), g^n-(g+5), \dots, g^n-(g(g-1)-1)$.
\end{proof}

\begin{theorem}\label{th2.18}
If $p$ is a prime number, $p>g-1$ and $n \in \N$, then $p^n$ is complete whenever the order of $g$, $o_g(p)$ is even. Otherwise, $p^n$ is complete provided that $g$ is a perfect square.
\end{theorem}

\begin{proof}
If $o_g(p)$ is even, then $o_g(p^n)$ is even for all $n\geq 1$, see Proposition \ref{prop2.19} below. 
Since $p$ is prime and greater than $g-1$, we have that $p$ and $g$ are relatively prime. It is well known that the equation $x^2 \equiv b(\mod p^n)$ has zero or two solutions. Let $a:=o_g(p^n)$. If $a$ is even, then we have $(g^{\frac{a}{2}})^2 \equiv 1(\mod p^n)$ so $(g^{\frac{a}{2}}) \equiv \pm 1(\mod p^n)$. Since $(g^{\frac{a}{2}}) \neq 1(\mod p^n)$ we get that $(g^{\frac{a}{2}}) \equiv -1(\mod p^n)$. The result follows from Theorem \ref{th2.12}. If $g$ is a perfect square and $a$ is odd, then $(g^{\frac{a+1}{2}})^2 \equiv g(\mod p^n)$. Therefore $(g^{\frac{a+1}{2}}) \equiv \pm \sqrt g (\mod p^n)$. If $g$ is a perfect square, $\sqrt g$ or $-\sqrt g$ is in $G_{m,g}$ and the result again follows from Theorem \ref{th2.12}.

\end{proof}

\begin{remark}\label{rem2.14}
There are prime numbers which are not complete. 
Consider $g=6$ and the prime number $ p = 55987$. Then $6^7 \equiv 1(\mod 55987)$, so the order of $6$ in $\bz_p^\times$, $o_6(55987)=7$ is odd. An extreme cycle for this digit set is $$\{311, 9383, 10895, 11147, 11189, 11196, 1866\},$$ so we see that $p$ is incomplete. 
\end{remark}
\subsection{Primitive numbers}

\begin{proposition}
A number $m$ is incomplete if and only if it is divisible by a primitive number. 
\end{proposition}

\begin{proof}
Suppose that $m$ is incomplete. Then either $m$ is primitive, and hence divisible by a primitive number, or $m$ is not primitive. If $m$ is incomplete and not primitive, then a proper divisor $d$ of $m$ must be incomplete. Similarly, either $d$ is primitive, or a proper divisor of $d$ is incomplete. Continuing this process until we run out of proper divisors, we find that a proper divisor of $m$ must be primitive.  

On the other hand, suppose that $m$ is divisible by a primitive number $p$. Since $p$ is incomplete, by Lemma \ref{lem2.5}, all odd multiples of $p$ are also incomplete, so $m$ is incomplete. 
\end{proof}

\begin{lemma}\label{lem2.8}
If $m$ is a primitive number for $g$, then $m$ and $g$ are relatively prime.
\end{lemma}

\begin{proof}
Suppose that $m$ is a primitive number and that $\gcd(m,g) = d$, with $d>1$. We know by Lemma \ref{lemi1.4} that there is an extreme cycle point in $\bz$, $x_0=\frac{g^{r-1}l_{r-1}+g^{r-2}l_{r-2}+\dots+gl_{1}+l_{0}}{g^r-1}$, with $l_k \in \{0,m\}$. Since each $l_k$ is either $0$ or $m$, where $m$ is divisible by $d$, and since $g^r-1$ is not divisible by any of the prime factors of $d$, we have that $x_0$ is also divisible by $d$. Dividing everything by $d$ we get that $x_0/d$ is an extreme cycle for $\{0,m/d\}$. But $m/d$ is complete, because $m$ is primitive, a contradiction. Thus $m$ and $g$ are relatively prime.
\end{proof}

\begin{theorem}\label{th2.17}
There are infinitely many primitive numbers. 
\end{theorem}

\begin{proof}
Suppose there are only finitely many primitive numbers and let $m_1,\dots, m_s$ be the ones bigger than $g-1$. By Lemma \ref{lem2.8}, the numbers $m_i$ are relatively prime with $g$ so the order $o_g(m_i)$ of $g$ in $U(\bz_{m_i})$ is well defined. Let $n$ be a common multiple of $o_g((g-1)^2)$, $o_g(m_1),\dots, o_g(m_s)$, larger than $g-1$. 

Then $g^{n+1}-1\equiv g-1\mod((g-1)^2,m_1,\dots,m_s)$. Let $m=\frac{g^{n+1}-1}{g-1}$. This is an odd number. We have that $m$ is not divisible by $g-1,m_1,\dots$ or $ m_s$, otherwise $g^{n+1}-1$ is divisible by $(g-1)^2$, $m_1,\dots $ or  $ m_s$. 
Consider the cycle point $x_0$ with digits $l_{0}=m,\dots, l_{g-2}=m, l_{g-1}=0,\dots, l_n=0$, as in Lemma \ref{lemi1.4}.  Then 
$$x_0=\frac{m(1+g+\dots+g^{g-2})}{g^{n+1}-1}=\frac{1+g+\dots+g^{g-2}}{g-1}$$
But $g\equiv 1(\mod (g-1))$ so $1+g+\dots+g^{g-2}\equiv \underbrace{1+1+\dots+1}_{\mbox{$g-1$ times}}\equiv g-1\equiv 0 (\mod (g-1))$. So $x_0\in \bz$. With Lemma \ref{lem2.7}, it follows that $m$ is incomplete, so it is divisible by a primitive number, contradiction. 
\end{proof}

\subsection{Properties of the order of a number}\label{sec2.4}

\begin{definition}
For a prime number $p \geq 3$, we denote by $\iota_g(p)$ the largest number $l$ such that $o_g(p^l) = o_g(p)$. We say that $p$ is $simple$ if $o_g(p) < o_g(p^2)$, i.e., $\iota_g(p) = 1$. 
\end{definition}

\begin{proposition}\label{prop2.17}
Let $m$ and $n$ be relatively prime odd integers. Then $$o_g(mn) = \lcm(o_g(m),o_g(n)).$$
\end{proposition}

\begin{proof}
We have that $a = o_g(mn)$ is the smallest integer such that $g^a \equiv 1(\mod mn)$. So $a$ is the smallest integer such that $g^a \equiv 1(\mod m)$ and $g^a \equiv 1(\mod n)$, which means that $a$ is the smallest integer that is divisible by $o_g(m)$ and $o_g(n)$ so it is the lowest common multiple of these two numbers. 
\end{proof}

\begin{lemma}\label{lemi2.18}
Let $p$ be an odd prime number relatively prime with $g$. Then $o_g(p^l) \leq o_g(p^{l+1})$.
\end{lemma}

\begin{proof}
Suppose to the contrary that $o_g(p^l) > o_g(p^{l+1})$. Let $a = o_g(p^l)$ and $b = o_g(p^{l+1})$, with $a > b$. Then we have that $g^a \equiv 1(\mod p^l)$ and $g^b \equiv 1(\mod p^{l+1})$, so $p^l \vert g^a-1$ and $p^{l+1} \vert g^b-1$. Since $p^{l+1} \vert g^b-1$, we also have that $p^{l} \vert g^b-1$. Thus $g^b \equiv 1(\mod p^l)$, which means that $a$ divides $b$. This contradicts the fact that $a>b$, so we have that $o_g(p^l) \leq o_g(p^{l+1})$.
\end{proof}

\begin{proposition}\label{prop2.19}
Let $p$ be an odd prime number relatively prime with $g$. Then $o_g(p^k) = o_g(p)$ for $k \leq \iota_g(p)$ and $o_g(p^k) = p^{k-\iota_g(p)}o_g(p)$ for all $k > \iota_g(p)$. 
\end{proposition} 

\begin{proof}
For $k \leq \iota_g(p)$, the statement follows from Lemma \ref{lemi2.18}. Assume by induction that for $k \geq \iota_g(p)$, $a_k := o_g(p^k) = p^{k-\iota_g(p)}o_g(p)$ and $o_g(p^k) < o_g(p^{k+1})$. Then there exists $q$ not divisible by $p$ such that $g^{a_k} = 1 + qp^k$. Raise this to power $p$ using the binomial formula: 
$$g^{pa_k} = 1 + p\cdot qp^k + q'p^{k+2},$$ for some integer $q'$. This implies that $a_{k+1} = o_g(p^{k+1})$ divides $pa_k$ and also that $pa_k$ is not $o_g(p^{k+2})$. Since $g^{a_{k+1}} \equiv 1(\mod p^{k+1})$ we have also that $g^{a_{k+1}} \equiv 1(\mod p^k)$ so $a_k$ divides $a_{k+1}$. Thus $a_{k+1}$ is a number that divides $pa_k$ and is divisible by $a_k$, and by the induction hypothesis $a_{k+1} > a_k$. Thus $a_{k+1} = pa_k = p^{k+1-\iota_g(p)}o_g(p)$. Also, $o_g(p^{k+1}) = pa_k \neq o_g(p^{k+2})$ so $o_g(p^{k+1}) < o_g(p^{k+2})$. Using induction we obtain the result. 
\end{proof}

\begin{proposition}\label{pr2.24}
Let $p_1, \dots, p_r$ be distinct odd primes relatively prime with $g$ and $k_1, \dots, k_r \geq 0$. For $i \in \{1,\dots,r\}$, let $j_i \geq 0$ be the largest integer such that $p_i^{j_i}$ divides $\lcm(o_g(p_1), \dots, o_g(p_r))$. Then 
\begin{equation}
 o_g(p_1^{k_1} \dots p_r^{k_r}) = \Bigg(\prod_{i=1}^{r} p_i^{\max\{k_i-j_i-\iota_g(p_i),0\}}\Bigg) \lcm(o_g(p_1), \dots, o_g(p_r)). 
\label{eq2.24.1}
\end{equation}

\end{proposition}

\begin{proof}
We have that $o_g(p_1^{k_1} \dots p_r^{k_r}) = \lcm(o_g(p_1^{k_1}), \dots, o_g(p_r^{k_r}))$ by Proposition \ref{prop2.17}. By Proposition \ref{prop2.19}, $$o_g(p_1^{k_1} \dots p_r^{k_r}) = \lcm\bigg(p_i^{\max\{k_i-\iota_g(p_i),0\}}o_g(p_i); i \in \{1, \dots, r\}\bigg).$$ If $k_i - \iota_g(p_i) \leq j_i$, then $p_i^{\max\{k_i-\iota_g(p_i),0\}}$ already divides $\lcm(o_g(p_1), \dots, o_g(p_r))$ so it does not contribute to the right-hand side. If $k_i-\iota_g(p_i) > j_i$, then $p_i^{\max\{k_i-\iota_g(p_i),0\}}$ contributes with $p_i^{k_i-\iota_g(p_i)-j_i}$ to the right-hand side. Then \eqref{eq2.24.1} follows. 
\end{proof}

\begin{proposition}\label{prop2.23}
Let $m$ be a primitive number and let $C = \{x_0, \dots, x_{p-1}\}$ be an extreme cycle. Then:
\begin{enumerate}
	\item The length $p$ of the cycle is equal to $o_g(m)$. 
	\item Every element of the cycle $x_i$ is mutually prime with $m$. 
	\item The extreme cycle $C$ is a coset of the group $G_{m,g}$ in $U(\Z_m)$, $C=x_0G_{m,g}$. 
	\item The number $m$ is primitive if and only if it is incomplete and $\gcd(C) = 1$ for all extreme cycles $C$.  
\end{enumerate}
\end{proposition}

\begin{proof}
Suppose $x_0$ and $m$ have a common divisor $d > 1$. Then, since $x_1 = \frac{x_0+l_0}{g}$ we have that $gx_1$ is divisible by $d$. From Lemma \ref{lem2.8}, we have that $g$ and $m$ are relatively prime because $m$ is primitive, so $d$ must divide $x_1$. By induction $d$ divides all the elements of the cycle. But then $\{\frac{x_0}{d}, \frac{x_1}{d}, \dots, \frac{x_{p-1}}{d}\}$ is an extreme cycle for the digits $\{0, \frac{m}{d}\}$. This contradicts that $m$ is primitive. 

We have $g^jx_i \equiv x_{(i-j)(\mod p)}(\mod m)$ for all $i, j \in \{0, \dots, p-1\}$. Therefore $g^px_0 \equiv x_0(\mod m)$. Since $x_0$ is in $U(\Z_m)$, we get that $g^p \equiv 1(\mod m)$, so $p$ is divisible by $o_g(m) =: a$. Also, we have $x_0 \equiv g^ax_0 \equiv x_{-a(\mod p)}(\mod m)$ so, since all the elements of the cycle are in $[0, \frac{m}{g-1}]$ we get that $x_0 = x_{-a(\mod p)}.$ Therefore $a$ is divisible by $p$. Thus $p = a = o_g(m)$. 

Since the length of the cycle is $o_g(m)$ which is the order of the group $G$, and since $g^jx_0(\mod m) = x_{-j(\mod p)}$, we get that $x_0G_{m,g} = C$.

For (iv), suppose that $k = \gcd(C) > 1$. Then, one of the digits for the cycle is $m$, we can assume it is the first one, therefore we have $x_0 + m = gx_1$, which implies that $k$ divides $m$. Thus $\{\frac{x_i}{k} : i = 0, 1, \ldots, p-1\}$ is a cycle for $\frac{m}{k}$, contradicting that $m$ is primitive.  

Conversely, suppose that $m$ is not primitive. Then there exists a primitive number $p$ such that $m=pk$, $k \in \mathbb{N}$. Then $p$ has an extreme cycle $C$. So $kC$ is an extreme cycle for $m$, but $\gcd(kC) \geq k$, a contradiction.

\end{proof}

\subsection{The order and possible cycles}

\begin{theorem}\label{th4.1}
The only primitive number of order 1 is $g-1$.
There are no primitive numbers of order 2 or 3. If $g-1$ is not divisible by $3$ then there are no primitive numbers of order 4 or of order 5. If $g-1$ is divisible by 3, then there exists a unique primitive number of order 4, namely $m=\frac{g^4-1}{3}$, and there exists a unique primitive number of order 5, namely $\frac{g^5-1}{3}$. 
\end{theorem}

\begin{proof}
The first statement is clear from Lemma \ref{lem2.4.1} and Lemma \ref{lem2.6}.
Suppose $m$ is a primitive number of order 2. Then, by Proposition \ref{prop2.23}, there exists an extreme cycle of length 2. The only possible digits that correspond to a cycle of length 2, up to a cyclical permutation, are $m0$. Then, by Lemma \ref{lemi1.4}, the cycle point is $x_0=\frac{m}{g^2-1}\in\bz$. This implies that $m$ is divisible by the primitive number $g-1$, which has order 1, a contradiction. 

Suppose now $m$ is a primitive number of order 3. Then it has an extreme cycle of length 3. The digits corresponding to such a cycle can be $000$, $00m$, $0m0$, $m00$, $0mm$, $mm0$ $m0m$ and $mmm$. The digits $000$ correspond to the trivial cycle $\{0\}$. The digits $mmm$ correspond to a cycle of length 1, not 3. The digits $00m$, $0m0$ and $m00$ correspond to three points in the same extreme cycle, and if one sequence appears then the other two appear too, therefore we can consider just one of them, e.g, $m00$. Same for $0mm$, $m0m$, $mm0$, we can consider just $mm0$.

Thus, up to a cyclical permutation, the only possible digits for such a cycle are, $m00$ or $mm0$. In the first case, the cycle point is $x_0=\frac{m}{g^3-1}$, and then $m$ is divisible by $g-1$, a contradiction. In the second case, the cycle point is $x_0=\frac{m(g+1)}{(g-1)(g^2+g+1)}$. But $g-1$ and $g+1$ are mutually prime (since $g$ is even) so $m$ is divisible by $g-1$, a contradiction. 

 Suppose $m$ is a primitive number of order 4. Then it has an extreme cycle of length 4. The digits for such a cycle can only be $m000$, $mm00$ or $mmm0$. In the first case the cycle point is $x_0=\frac{m}{g^4-1}$ so $m$ is divisible by $g-1$, contradiction. In the second case, the cycle point is $x_0=\frac{m(g+1)}{g^4-1}$. Since $g-1$ and $g+1$ are mutually prime, it follows that $m$ is divisible by $g-1$. In the last case the cycle point is $x_0=\frac{m(1+g+g^2)}{g^4-1}=\frac{m(1+g+g^2)}{(g-1)(g+1)(g^2+1)}$. If a prime number $p$ divides both $1+g+g^2$ and $g+1$ then it has to divide $g^2$ so it divides $g$ and $g+1$ so it divides 1. Therefore $1+g+g^2$ and $g+1$ are mutually prime so $m$ is divisible by $g+1$. If a prime number divides both $1+g+g^2$ and $g^2+1$ then it must divide $g$ so it divides 1, so $1+g+g^2$ and $g^2+1$ are mutually prime and therefore $m$ is divisible by $g^2+1$. If a prime number $p$ divides both $g^2+g+1$ and $g-1$ then it divides $g^2-2g+1$ so it divides $3g$. Then, either $p=3$ or $p$ divides $g$. If $p$ divides $g$ then it divides $1$. Thus the only common divisor of $g-1$ and $g^2+g+1$ can be 3. If $g-1$ is not divisible by 3, then $1+g+g^2$ and $g-1$ are mutually prime so $m$ is divisible by $g-1$, a contradiction. If $g-1$ is divisible by 3, then $g=3k+1$ for some $k\in\bz$ and so 
$1+g+g^2=3(1+3k+3k^2)$. This means that $1+g+g^2$ is not divisible by $9$ and therefore the greatest common divisor of $1+g+g^2$ and $g-1$ is 3. Then $m$ has to be divisible by $\frac{g-1}{3}\times (g+1)\times (g^2+1)=\frac{g^4-1}{3}$. 

Note that the number $\frac{g^4-1}{3}$ is incomplete since it has an extreme cycle point with digits $mmm0$. If it is not primitive, then there is a primitive number $m$ which divides it. Then $m$ divides $g^4-1$ so $g^4\equiv 1(\mod m)$ and therefore $o_g(m)$ divides 4, hence the order of $m$ is either 1,2 or 4. We ruled out the first two cases. If the the order of $m$ is 4, then from the discussion above, it follows that $m$ is divisible by $\frac{g^4-1}{3}$. So $m=\frac{g^4-1}{3}$. 

Suppose now $m$ is a primitive number of order 5. Then it has an extreme cycle of length 5. The digits for such a cycle can only be: $m0000$, $mm000$, $m0m00$, $mmm00$, $mm0m0$, $mmmm0$. 

For $m0000$ the cycle point is $x_0=\frac{m}{g^5-1}$ so $m$ is divisible by $g-1$, a contradiction. 

For $mm000$ the cycle point is $x_0=\frac{m(1+g)}{g^5-1}$. Since $g-1$ and $g+1$ are mutually prime, it follows that $m$ is divisible by $g-1$, contradiction. 

For $m0m00$ the cycle point is $x_0=\frac{m(1+g^2)}{g^5-1}$. If a prime number divides both $1+g^2$ and $g-1$ then it divides $g^2-2g+1$, so it divides $2g$, so it divides $g$, so it divides 1. Therefore $g-1$ and $1+g^2$ are mutually prime so $m$ is divisible by $g-1$, contradiction.

For $mmm00$ the cycle point is $x_0=\frac{m(1+g+g^2)}{g^5-1}$. If a prime number $p$ divides both $1+g+g^2$ and $g-1$ then, as in the discussion for the case of order 4, we get that $p=3$ and $g-1$ has to be divisible by $3$ and $\gcd(1+g+g^2,{g-1})=3$. Also, if a prime number divides both $1+g+g^2$ and $1+g+g^2+g^3+g^4$ then it divides $g^3(g+1)$ so it either divides $g$ or it divides $g+1$. If it divides $g$ then it divides $1$, and if it divides $g+1$ then it divides $g^2$, so it divides $g$, so it divides $1$. Thus, $1+g+g^2$ and $1+g+\dots+g^4$ are mutually prime and therefore $m$ is divisible by $1+g+\dots+g^4$. Hence $m$ is divisible by $\frac{g-1}3\times (1+g+\dots+g^4)=\frac{g^5-1}{3}$.

For $mm0m0$ the cycle point is $x_0=\frac{m(1+g+g^3)}{g^5-1}$. If a prime number $p$ divides both $1+g+g^3$ and $g-1$ then it divides $g^3-g^2$ so it divides $1+g+g^2$, then as before, $p=3$ and $g-1$ is divisible by 3. We prove that $\gcd(1+g+g^3, g-1)=3$. As we saw, the only prime number that divides both $1+g+g^3$ and $g-1$ is $3$; so we have to show only that $9$ does not divide both numbers. Let $g=3k+1$ with $k\in\bz$. Then $1+g+g^3=3(1+4k+9k^2+9k^3)$. If $9$ divides $1+g+g^3$ then 3 divides $1+k$, so $k=3l+2$ for some $l\in\bz$. But then $g-1=3(3l+2)=9k+6$ is not divisible by $9$. Thus $\gcd(1+g+g^3,{g-1})=3$. 

If a prime number divides both $1+g+g^3$ and $1+g+\dots+g^4$, then it divides $g^2(1+g^2)$, so it either divides $g$ or it divides $1+g^2$. If it divides $g$ then it divides $1$, and if it divides $1+g^2$ then it divides $g+g^3$ so it divides 1. Thus, $1+g+g^3$ and $1+g+\dots+g^4$ are mutually prime. Therefore 
$m$ has to be divisible by $\frac{g-1}{3}\times (1+g+\dots+g^4)=\frac{g^5-1}{3}$.

For $mmmm0$ the cycle point is $x_0=\frac{m(1+g+g^2+g^3)}{g^5-1}$. If a prime number divides $1+g+g^2+g^3$ and $g-1$ then it divides $g^3-g^2$, so it divides $1+g+2g^2$ and $2-4g+2g^2$ so it divides $5g-1$ and $5g-5$, so it divides $4$, which is impossible because $g-1$ is odd. Thus $m$ has to be divisible by $g-1$, contradiction.

In conclusion, if $g-1$ is not divisible by 3, then there are no primitive numbers of order 5. If $g-1$ is divisible by $3$, then a primitive number of order 5 must be divisible by $\frac{g^5-1}{3}$. This number is incomplete because it has at least two extreme cycles with digits $mmm00$ and $mm0m0$. If it is not primitive, then it is divisible by a primitive number $m$. Then $m$ divides $g^5-1$ so $g^5\equiv 1(\mod m)$ so the order of $m$ divides 5. We cannot have $o_g(m)=1$ so $o_g(m)=5$. From the previous discussion, we obtain that $m$ is divisible by $\frac{g^5-1}{3}$ so $m=\frac{g^5-1}{3}$.

\end{proof}

\begin{theorem}\label{th4.8.1}
Let $g=p+1$ where $p$ is a prime number. Then there are no non-trivial primitive numbers of order strictly less than $g$. 
\end{theorem}

\begin{proof}
Let $m$ be a non-trivial primitive number of order $n$. Then, by Proposition \ref{prop2.23}, it has an extreme cycle of length $n$ with some digits $l_0,\dots,l_{n-1}\in\{0,m\}$. Let $k_i:=l_i/m\in\{0,1\}$, for $i\in\{0,\dots, n-1\}$. The cycle point is
$$x_0=\frac{m(k_0+gk_1+\dots+g^{n-1}k_{n-1})}{g^n-1}=\frac{m(k_0+gk_1+\dots+g^{n-1}k_{n-1})}{p(1+g+\dots+g^{n-1})}\in\bz.$$

Since $m$ is a non-trivial primitive number, it cannot be divisible by $g-1=p$. Therefore $k_0+gk_1+\dots+g^{n-1}k_{n-1}$ must be divisible by $p$. However $g\equiv1(\mod p)$ so $g^k\equiv 1(\mod p)$ for all $k$. Then $k_0+gk_1+\dots+g^{n-1}k_{n-1}\equiv k_0+k_1+\dots+k_{n-1}(\mod p)$, so $k_0+\dots+k_{n-1}$ must be divisible by $p$. 
Therefore we must have a multiple of $p$ ones among the digits $k_0,\dots,k_{n-1}$, so we have at least $p$ ones. Also, not all the digits can be 1, because then 
$x_0=m/p$ so $m$ is divisible by $g$, a contradiction. Therefore we must have at least $p+1=g$ digits, so $n\geq g$. 
\end{proof}

\begin{theorem}\label{th4.9}
Let $m$ be a non-trivial primitive number, $o_g(m)=:n$ and let $x_0$ be an extreme cycle point with digits $l_0,\dots, l_{n-1}$ as in \eqref{eqec}. Let $k_i:=l_i/m\in\{0,1\}$ for $i\in\{0,\dots,n-1\}$ and let $d:=\gcd(k_0+gk_1+\dots +g^{n-1}k_{n-1}, g^n-1)$. Then $m=\frac{g^n-1}{d}$.

Also, if $k_0,\dots, k_{n-1}$ are some digits in $\{0,1\}$ and if $d:=\gcd(k_0+gk_1+\dots +g^{n-1}k_{n-1}, g^n-1)$ then the number $m:=\frac{g^n-1}{d}$ is incomplete and has an extreme cycle with digits $mk_0,mk_1,\dots,mk_{n-1}$. 
\end{theorem}

\begin{proof}
First note that we know that the length of the cycle is equal to $n$, from Proposition \ref{prop2.23}.
With Lemma \ref{lemi1.4}, we have that 
\begin{equation}
x_0=\frac{m(k_0+gk_1+\dots+g^{n-1}k_{n-1})}{g^n-1}=\frac{m\frac{k_0+gk_1+\dots+g^{n-1}k_{n-1}}{d}}{\frac{g^n-1}d}
\label{eq4.9.1}
\end{equation}
But $\frac{k_0+gk_1+\dots+g^{n-1}k_{n-1}}{d}$ and $\frac{g^n-1}d$ are mutually prime, and since $x_0$ is an integer, it follows that $m$ must be divisible by $\frac{g^n-1}d$. Let $m':=\frac{g^n-1}d$. Then 
$$x_0':=\frac{m'(k_0+gk_1+\dots+g^{n-1}k_{n-1})}{g^n-1}=\frac{k_0+gk_1+\dots+g^{n-1}k_{n-1}}{d}$$
is a cycle point for the digits $\{0,m'\}$ and it is in $\bz$, therefore, by Lemma \ref{lem2.7}, it is an extreme cycle point for $m'$. This means that $m'$ is incomplete. Since $m$ is divisible by $m'$ and it is also primitive, it follows that $m=m'$. 

The last statement of the theorem follows from the previous computations. 
\end{proof}

\begin{example}Recall from \cite{DuHa16}, that the first few primitive numbers for $g=4$ are $$\{3, 85, 341, 455, 1285, 4369, 5461\}.$$ They can be obtained very nicely as: 
$$3 = 4^1-1, 85 = \frac{4^4-1}{3}, 341 = \frac{4^5-1}{3}, 455 = \frac{4^6-1}{3^2}, 1285 = \frac{4^8-1}{3\cdot 5 \cdot 17}, 4369 = \frac{4^8-1}{3\cdot 5}, 5461 = \frac{4^7-1}{3}.$$
\end{example}

\begin{corollary}\label{cor4.9}
All primitive numbers $m$ are divisors of ${g^n-1}$, where $o_g(m) = n$.
\end{corollary}

\begin{example}\label{ex4.10}
We illustrate how we can use Theorem \ref{th4.9} to find some non-trivial primitive numbers. Take for example $g=16$. We want a non-trivial primitive number $m$ so $m$ cannot be divisible by $g-1=15$. Also, it must have an extreme cycle, so for some choice of digits $k_0,\dots,k_{n-1}\in\{0,1\}$ we must have that 
$$x_0:=\frac{m(k_0+16k_1+\dots+16^{n-1}k_{n-1})}{16^n-1}$$
is an integer. Since $16^n-1$ is divisible by 15, the denominator must be divisible by 15. But $m$ should not be divisible by 15. So the term $k_0+16k_1+\dots+16^{n-1}k_{n-1}$ must contain some factors of 15, i.e., 3 or 5.

 Let's pick 3 first. Since $k_0+16k_1+\dots+16^{n-1}k_{n-1}\equiv k_0+k_1+\dots+k_{n-1}(\mod 15)$ (and $(\mod 3)$ and $(\mod 5)$), we must have $k_0+\dots+k_{n-1}$ divisible by 3. Therefore we have a multiple of 3 number of ones among these digits. We cannot just pick $111$ because that is actually the cycle with digit $1$. So we can pick $1110$. Thus $n=4$. Then $k_0+16k_1+\dots+16^{n-1}k_{n-1}=1+16+16^2$ is divisible by 3. We take $m=\frac{16^4-1}{3}$ and using Theorem \ref{th4.9}, or by a direct check we can see that the number is primitive. 

 We can do a similar thing for 5. We must have $k_0+\dots+k_{n-1}$ divisible by 5, so we need at least 6 digits, such as $111110$. Then we take $m=\frac{16^6-1}{5}=3355443$. A computer check shows that the only extreme cycle is $\{ 13981,  210589, 222877,  223693,  223645,  223696\}$ and these numbers are relatively prime. Therefore, with Proposition \ref{prop2.23}, we obtain that this number is primitive too.

Now let's take $g=12$. A non-trivial primitive number $m$ cannot be divisible by $g-1=11$. Therefore, we must find digits so that $k_0+12k_1+\dots+12^{n-1}k_{n-1}$ is divisible by $11$. As before, this implies that $k_0+\dots+k_{n-1}$ is divisible by 11, so we must have a multiple of 11 number of ones among these digits! We need some large numbers! We can take 
$\underbrace{11\dots1}_{11\mbox{ times}}0$. So $n=12$. We pick $m=\frac{12^{12}-1}{11}=810554586205$. A computer check shows that the only extreme cycle is 
$\{68057929271,  73217709623,  73647691319,  73686509111, $ 
$ 73683523127,  73686778679,  73686757943,$ $  
73686780563,  73686780564,  73686780407, 73686780551, 6140565047\}$.
 The numbers are relatively prime, and by Proposition \ref{prop2.23}, it follows that this number $m$ is primitive. 

\end{example}

\begin{lemma}\label{lem4.11}
The prime divisors of $g^n-1$ are precisely the prime numbers with order dividing $n$. 
\end{lemma}

\begin{proof}
Let $p$ be a prime number with $o_g(p) = l$, and $l|n$. Since $o_g(p) = l$, we have that $g^l \equiv 1(\mod p)$. Since $l | n$, we have that $n = lj$, for some $j \in \mathbb{Z}$. Thus, 
$$ (g^l)^j \equiv 1^j(\mod p) \implies g^n \equiv 1(\mod p) \implies g^n-1 \equiv 0(\mod p).$$
So, we have that $p | g^n-1$. 

Conversely, if $p$ is a prime divisor of $g^n-1$ then $g^n\equiv 1(\mod p)$ so $o_g(p)$ divides $n$.  
\end{proof}

\begin{theorem}\label{th2.23}
Let $q> g-1$ be mutually prime with $g-1$. Then $m:=\frac{g^q-1}{g-1}$ is incomplete and $o_g(m)=q$. All divisors $e>1$ of $m$ have $o_g(e)\neq 1$ and $o_g(e) | q$. 
If, in addition, $q$ is prime, then there exist primitive numbers of order $q$ and all primitive numbers $d$ that divide $m$ have $o_g(d)=q$. 
\end{theorem}

\begin{proof}
We know from Lemma \ref{lem4.11} that, for all prime divisors $d$ of $g^q-1$, $o_g(d)$ divides $q$. We have the factorization 
$g^q-1=(g-1)m$. We prove that $g-1$ and $m$ are mutually prime. If a prime number $p$ divides both $g-1$ and $m$, then $g\equiv 1(\mod p)$ so $g^n\equiv 1(\mod p)$ for all $n\in\bn$. So $m=1+g+\dots+g^{q-1}\equiv 1+1+\dots+1=q(\mod p)$. But $p$ divides $m$ so $0\equiv q(\mod p)$ which means that $p$ divides $q$, and this contradicts the fact that $g-1$ and $q$ are mutually prime. 

We show that if $e>1$ divides $m$ then $o_g(e)\neq 1$. If not, then $g\equiv 1(\mod e)$ so $e$ divides $g-1$. But $e$ divides $m$, and $g-1$ and $m$ are mutually prime, a contradiction. 

Clearly we have that $m$ divides $g^q-1$ so $g^q\equiv 1(\mod m)$ so $o_g(m)$ divides $q$. For $1\leq l<q$, $0<g^l-1<m$ so $g^l-1\not\equiv 0(\mod m)$. Thus $o_g(m)=q$. Therefore, any divisor of $e>1$ of $m$ has $o_g(e) | q$. 

Next, we show that $m$ is incomplete. Consider the cycle point $x_0$ with digits $$\underbrace{m,m,\dots, m}_{g-1\mbox{ times}}, \underbrace{0,0,\dots,0}_{q-g+1\mbox{ times}}.$$ 
Then, by Lemma \ref{lemi1.4}, we have
$$x_0=\frac{m(1+g+\dots+g^{g-2})}{g^q-1}=\frac{m(1+g+\dots+g^{g-2})}{(g-1)m}=\frac{1+g+\dots+g^{g-2}}{g-1}.$$
We have $g\equiv 1(\mod (g-1))$ so $g^l\equiv 1(\mod (g-1))$. Then $1+g+\dots +g^{g-2}\equiv (g-1)\equiv 0(\mod (g-1))$. So $x_0$ is an integer and therefore an extreme cycle point. So $m$ is incomplete. 

Assume now that $q$ is prime. Since $m$ is incomplete, there exists a divisor $m$ which is a primitive number. Then, $d$ divides $g^q-1$ so, by Lemma \ref{lem4.11}, $o_g(d)$ divides $q$ so it is 1 or $q$. However, it cannot be 1, since that would imply that $d$ divides $g-1$, and since $d$ divides $m$, this would contradict the fact that $g-1$ and $m$ are mutually prime. Therefore $o_g(d)=q$. 

\end{proof}

\begin{remark}\label{rem4.10}
The condition that $q$ is prime cannot be removed, if we want to find a primitive number of order $q$. For example,
there is no primitive number of order $q=14$ for $g=6$. We have that $14$ and $g-1=5$ are mutually prime. Also, we have that $6^{14}-1 = 5\cdot 7 \cdot 7 \cdot 29 \cdot 197 \cdot 55987$. Since $5$ and $55987$ are primitive for this $g$, of order $1$ and $7$ respectively, a primitive number of order 14 would have to be a divisor of $7 \cdot 7 \cdot 29 \cdot 197 = 279937$. However, this number is complete. 
\end{remark}

\begin{remark}
Theorem \ref{th2.23} can be used in finding new primitive numbers. When $g=4$, we know that prime numbers cannot be primitive. The following numbers must all be primitive because they are of prime order and the product of exactly two prime numbers (and all prime numbers are complete for $g=4$, by Theorem \ref{th2.18}):

$$\frac{4^{13}-1}{3} = 22369621 =2731 \cdot 8191$$ 
$$\frac{4^{17}-1}{3} = 5726623061 =43691 \cdot 131071 $$
$$\frac{4^{19}-1}{3} = 91625968981 =174763 \cdot 524287$$ 

However, 
$$\frac{4^{23}-1}{3} = 23456248059221 = 47 \cdot 178481 \cdot 2796203$$
is merely incomplete. A computer check shows that $8388607 = 47 \cdot 178481$ is complete, while $131421541 = 47 \cdot 2796203$ and $499069107643=178481 \cdot 2796203$ are primitive. 

\end{remark}

\begin{remark}
It is possible for $\frac{g^n-1}{g-1}$ to be complete. Take $g=22$ and $n=7$. Then $\frac{22^7-1}{21}=118778947$  is complete.
\end{remark}

\begin{corollary}\label{cor2.24}
Let $g=p+1$ where $p$ is a prime number. Then there are no non-trivial primitive numbers of order strictly less than $g$ and, for every prime number $q>g$, there exists a primitive number of order $q$. 
\end{corollary}

\begin{proof}
The first part is contained in Theorem \ref{th4.8.1}, and the second part follows immediately from Theorem \ref{th2.23}. 
\end{proof}

\begin{example}
This example demonstrates a possible method for determining whether there exists a primitive number of order $n$. Let $g=4$, since $g-1=3$ is prime, it has already been shown that a primitive number exists for every prime $q>4$. We now consider multiples of prime numbers.   
Consider $n=22$. There are no primitive numbers of order $2$, and the only primitive number of order $11$ is $60787 = 89 \cdot 683$. Using the relationship between cycle points, and assuming, without loss of generality, that the last two digits in the cycle are $m0$, we have that 

$$x_0 = \frac{4m(k_0 + k_1\cdot 4 + \ldots + k_{19}\cdot 4^{19} + 4^{20})}{4^{22}-1} = \frac{4m(k_0 + k_1\cdot 4 + \ldots + k_{19}\cdot 4^{19} + 4^{20})}{3(5\cdot 23 \cdot 89 \cdot 397 \cdot 683 \cdot 2113)}$$
for some $k_0,\dots,k_{19}$ in $\{0,1\}$.

The orders of the numbers in the denominator are $1, 2, 11, 11, 22, 11, 22$ respectively. In order for a primitive number of order $22$ to exist, we need to cancel $89$ and/or $683$ with the parenthesis in the numerator. Since the parenthesis in the numerator must also be divisible by $3$, we know we need exactly $3l-1$ terms in addition to the $4^{20}$ term. Consider the multiplicative groups generated by $4$ modulo $89$ and $683$, since our primitive number $m$ should not be divisible by $60787 = 89 \cdot 683$ which is primitive. 

For $89$, we have $\{4, 16, 64, 78, 45, 2, 8, 32, 39, 67, 1\}$ and $4^{20} \equiv 39 (\mod 89)$.
 
For $683$, we have $\{4, 16, 64, 256, 341, 681, 675, 651, 555, 171, 1\}$ and $4^{20} \equiv 555 (\mod 683)$.

We need to pick exactly $2$, $5$, or $8$ terms from these groups, add them together with $4^{20}$, and try to get a number equivalent to $0 (\mod 89 \text{ or } 683)$.  

Using a computer, we see that from the first set, $4 + 16 + 78 + 2 + 39 + 39 = 178 \equiv 0(\mod 89)$ and from the second set, $256 + 555 + 555 =  1366 \equiv(0 \mod 683)$ satisfy these conditions. 

So, for the numerator, we get $4 + 4^2 + 4^4 + 4^6 + 4^9 + 4^{20}$ in the first case and $4^4 + 4^9 + 4^{20}$ in the second. 

Thus the number $5\cdot 23\cdot 89\cdot 397\cdot  2113$ is incomplete. 
A computer check shows that $\frac{4^{22}-1}{3\cdot 5\cdot 683}$ is primitive. 

Also the number $5\cdot 23\cdot 397 \cdot 683\cdot 2113$ is incomplete. A computer check shows that $\frac{4^{22}-1}{3\cdot 5\cdot 89}$ is primitive. Both have order 22. 
\end{example}

 For the next theorem, when we say $x=d_0d_1\dots d_n$ in base $g$, we mean 
$$x=d_0g^n+d_1g^{n-1}+\dots+d_{n-1}g + d_n.$$

\begin{theorem}\label{th2.34}
Let $m = \underbrace{11 \ldots 1}_{g\text{-times}}$ in base $g$, so $m=\frac{g^g-1}{g-1}$. Then $m$ is primitive  with the base $g$ extreme cycle point $12 \ldots (g-2)(g-1)$ and $m$ has has cycle length $g$. Moreover, the cycle generated by this cycle point is the only extreme cycle for $m$. 
\end{theorem}

\begin{proof}
Note that all operations are taking place in base $g$. 
Let $x_0 = 123\ldots(g-3)(g-2)(g-1).$ Then
$$x_1 = \frac{123\ldots(g-3)(g-2)(g-1) + \overbrace{11 \ldots 1}^{g\text{-times}}}{g} = 123\ldots(g-4)(g-3)(g-1)0$$
$$x_2 = 123\ldots(g-4)(g-3)(g-1)$$
$$x_3 = \frac{123\ldots(g-4)(g-3)(g-1) + \overbrace{11 \ldots 1}^{g\text{-times}}}{g} = 1123\ldots(g-4)(g-3)(g-1)$$
$$x_4 = \frac{1123\ldots(g-4)(g-3)(g-1) + \overbrace{11 \ldots 1}^{g\text{-times}}}{g} = 1223\ldots(g-4)(g-3)(g-1)$$
$$\vdots$$
$$x_n = \frac{123\ldots(n-3)(n-3)\ldots(g-4)(g-3)(g-1) + \overbrace{11 \ldots 1}^{g\text{-times}}}{g} $$$$= 123\ldots(n-2)(n-2)\ldots(g-4)(g-3)(g-1)$$
$$\vdots$$ 
$$x_g = \frac{123\ldots(g-4)(g-3)(g-3)(g-1) + \overbrace{11 \ldots 1}^{g\text{-times}}}{g} = 123\ldots(g-3)(g-2)(g-1)$$

Since $x_g = x_0$, we have that this is indeed an extreme cycle of length $g$.

We prove that this is the only extreme cycle for $m$. Note that if $x_0$ has some decomposition $x_0=a_p\dots a_0=a_pg^p+\dots+a_1 g+a_0$ in base $g$ then the next element in the cycle is either $x_0/g$ or $(x_0+m)/g$. In the first case, the last digit $a_0$ has to be $0$ in the second case $a_0$ has to be $g-1$. 

In the case the last digit $a_0$ is $0$ we simply divide by $g$ and this means that in the base $g$ representation the last $0$ is removed, and we do so as many times this is possible, i.e., as many zeros we have in the end of the base $g$ representation, so we ignore the last zeroes and, for simplicity we talk about the cycle points that have an expansion that ends in a non-zero digit. 

Assume now the last digit $a_0$ is $g-1$ and consider the next to last digit $a_1$. The next element in the cycle is $x_1=(x+m)/g$. 

For a positive integer $x$ we will write $x=\dots a_ra_{r-1}\dots a_1a_0$ to indicate that the base $g$ representation ends in $a_ra_{r-1}\dots a_1a_0$. 

Since $x_0=\dots a_1 (g-1)$ and $m=\dots 11$, we get that $x_0+m=\dots ((a_1+2)\mod g)0$ and $x_1=\dots ((a_1+2)\mod g)$. Since $x_1$ is also a cycle point, its last digit is $0$ or $g-1$ therefore $a_1=g-2$ or $a_1=g-3$. 

We claim that every extreme cycle point for $m$ has the form 
\begin{equation}
\overbrace{1\dots1}^{n_1\text{-times}}\overbrace{2\dots2}^{n_2\text{-times}}\dots \overbrace{(g-2)\dots(g-2)}^{n_{g-2}\text{-times}}(g-1)
\label{eq2.28.1}
\end{equation}
with $n_1,\dots,n_{g-3}\geq 1$, $n_{g-2}\geq 0$.

First, we will prove that $x_0=\dots (g-3)(g-2)(g-2)\dots (g-2)(g-1)$ or $x_0=\dots (g-3)(g-1)$. If the next to last digit is $a_1=g-3$, we are done. If the next to last digit is $g-2$ we consider the digit immediately before it $a_2$. Since $x_0=\dots a_2(g-2)(g-1)$ we have 
$x_0+m=\dots ((a_2+2)\mod g)00$ so $x_1=\dots \dots ((a_2+2)\mod g)0$ and $x_2=\dots ((a_2+2)\mod g)$. Since this is an extreme cycle point, the last digit is either $0$ or $g-1$. Thus $a_2=g-2$ or $a_2=g-3$. By induction if $x_0=\dots a_l (g-2)\dots (g-2)(g-1)$ then $x_0+m=\dots ((a_l+2)\mod g)0\dots 00$, so dividing by $g$ as many times as needed we get an extreme cycle point of the form $\dots ((a_l+2)\mod g)$ and since the last digit has to be $0$ or $g-1$ it follows that $a_l=(g-2)$ or $a_l=(g-3)$. 

We show that we cannot have $x_0=(g-2)\dots (g-2)(g-1)$, so the digit $(g-3)$ has to appear. 

Note first that, by Proposition \ref{lemi1.4}, $x_0\leq \frac m{g-1}=\frac{g^{g-1}+\dots +g +1}{g-1}=\frac{g^n-1}{(g-1)^2}< g^{n-1}$ so $x_0$ has at most $g-1$ digits, so it has a shorter expansion than $m$ which has $g$ digits. 

If $x_0=(g-2)\dots (g-2)(g-1)$ then $x_0+m$ has the form $11\dots 120\dots 00$, which would imply that an extreme cycle point is of the form $11\dots 12$, a contradiction to the fact that the last digit has to be $0$ or $g-1$. 

Thus $x_0$ is of the form $\dots (g-3)(g-2)\dots (g-2)(g-1)$ and $g-2$ does not have to appear. Assume by induction that all extreme cycle points  $x_0$ (which do not end in 0) are of the form 
$$\dots a_l\overbrace{(g-k)\dots(g-k)}^{n_{g-k}\text{-times}}\dots \overbrace{(g-2)\dots(g-2)}^{n_{g-2}\text{-times}}(g-1),$$
with $k\geq 3$, $n_{g-k},\dots, n_{g-3}\geq 1$ and $n_{g-2}\geq 0$. Then $x_0+m=\dots(a_l+1)(g-k+1)\dots (g-2)\dots (g-2)(g-1)0\dots0$. Dividing by $g$ we get that an extreme cycle point is of the form $\dots(a_l+1)(g-k+1)\dots (g-2)\dots (g-2)(g-1)$, and by the induction hypothesis we obtain that $a_l+1=g-k+1$ or $a_l+1=g-k$ so 
$a_l=g-k$ or $a_l=g-k-1$. 

Thus the digits in the base $g$ expansion of $x_0$ form an increasing sequence and two consecutive digits differ by at most 1, with the exception of the last two which can be $(g-3)(g-1)$. 

We show that the first digit has to be $1$. Suppose $x_0=a_{p-1}\dots a_0$. We saw above that $x_0$ has at most $n-1$ digits then 
$x_0+m=1(a_{p-1}+1)\dots0$ so $x_1=1(a_{p-1}+1)\dots$. But we know that two consecutive digits of $x_1$ differ by at most $1$ so $a_{p-1}=1$. 

Combining these results we get that every extreme cycle point must have the form in \eqref{eq2.28.1}. 

Next we claim that either $n_1=\dots=n_{g-2}=1$ or $n_{g-2}=0$ and all but one of the $n_1,\dots,n_{g-3}$ are equal to $1$ with possibly at most one exception, which is equal to 2. 

Suppose first $n_{g-2}=0$. We know that the first digit is 1 and the last digits are $(g-3)(g-1)$. Also two consecutive digits before the $(g-3)$ differ by at most one and they appear in increasing order in the expansion. This means that all digits $1,2,\dots, (g-3)$ have to appear in the expansion (otherwise there is a jump by at least 2). So $n_1,\dots, n_{g-3}\geq 1$. 

On the other hand there are at most $g-1$ digits so $g-1\geq n_1+\dots +n_{g-3}+1\geq g-2$. This implies that we cannot have two numbers $n_i$ bigger than 2, moreover, at most one of them is 2 and the rest are $1$. 

If $n_{g-2}\geq 1$ then, with the previous argument, we get that all digits between $1$ and $g-2$ must appear in the expansion and then, as before we get $x_0=12\dots (g-1)$. Going through all the cases, we see that every possibility yields a point in the extreme cycle listed in the first part of the proof. 

We prove that $d=\gcd(C)=1$. Since $d$ divides $x_0=12\dots (g-3)(g-2)(g-1)$ and $gx_2=12\dots (g-3)(g-1)0$ it will divide also $gx_2-x_0=(g-1)g -((g-2)g+(g-1))=1$. 
\end{proof}

\begin{conjecture}\label{con3.29}
Let $m = \underbrace{11 \ldots 1}_{g\text{-times}}$ in base $g$, and let $g = p+1$ where $p$ is a prime number. Then $m$ is the first non-trivial primitive number.
\end{conjecture}

\begin{remark}By Theorem \ref{th2.34}, we have that $m$ is primitive. It remains to be shown that no primitive numbers can exist between $p$ and $m$. 
\end{remark}

\begin{example}\label{ex3.30}
Let us illustrate, with an example, an algorithm for finding primitive numbers. Let $g=6$. Of course, the trivial primitive number is $5$. Therefore, no other primitive number has 5 in its prime decomposition. 

By Corollary \ref{cor4.9}, the primitive numbers are divisors of $6^n-1$, and since we can remove the 5 from the prime decomposition, they have to be divisors of $\frac{6^n-1}{5}$. By Theorem \ref{th4.1}, we can start with $n=6$. When $n$ is not divisible by $g-1=5$, we can use Theorem \ref{th2.23} to conclude that $\frac{6^n-1}{5}$ is incomplete. 

By Theorem \ref{th2.34}, $\frac{6^6-1}{5}=7\cdot 31\cdot 43$ is primitive. 

We used a computer program to check if the numbers are complete or not.

For $n=7$, we have $\frac{6^7-1}{5}=55987$ is prime and incomplete, thus primitive. 

For $n=8$, $a=\frac{6^8-1}{5}=7\cdot 37\cdot 1297$ is incomplete. We checked that $\frac{a}{7},\frac{a}{37},\frac{a}{1297}$ are complete, therefore $\frac{6^8-1}{5}$ is primitive.

For $n=9$, $a=\frac{6^9-1}{5}=19\cdot 43\cdot 2467$ is incomplete. We checked that $\frac{a}{19},\frac{a}{43},\frac{a}{2467}$ are complete, therefore $\frac{6^9-1}{5}$ is primitive. 

For $n=10$, $a=\frac{6^{10}-1}{5}=5\cdot 7\cdot 11\cdot 101\cdot 311$. We have to remove the extra 5 from the prime decomposition. We checked that 
$\frac{a}{5\cdot 7},\frac{a}{5\cdot 11}, \frac{a}{5\cdot 101},\frac{a}{5\cdot 311}$ are complete, therefore $\frac{6^{10}-1}{5\cdot 5}$ is primitive.

For $n=11$, $a=\frac{6^{11}-1}{5}=23\cdot 3154757$. We checked that 23 is complete and 3154757 is prime and incomplete, therefore primitive. So $\frac{6^{11}-1}{5\cdot 23}$ is primitive. 

For $n=12$, $a=\frac{6^{12}-1}{5}=5\cdot 7\cdot 13\cdot 31\cdot 37\cdot 43\cdot 97$. We know that $7\cdot 31\cdot 43=\frac{6^6-1}{5}$ is primitive so at least on of these factors have to be removed. We checked that $\frac{a}{7},\frac{a}{31}$ are incomplete and $\frac{a}{43}$ is complete. Thus we cannot remove the factor 43 to get a primitive number. Then $\frac{a}{7\cdot 13},\frac{a}{7\cdot 37},\frac{a}{7\cdot 97}$ are complete and $\frac{a}{7\cdot 31}$ is incomplete. Also 
$\frac{a}{13\cdot 31},\frac{a}{13\cdot 97},\frac{a}{31\cdot 37},\frac{a}{31\cdot 97}$ are complete. This implies that $\frac{a}{7\cdot 31}=\frac{6^{12}-1}{5\cdot 7\cdot 31}$ is primitive, and this is the only divisor of $a$ (other than $7\cdot 31\cdot 43$) which is primitive. 

For $n=13$, $a=\frac{6^{13}-1}{5}=760891\cdot 3443$. Both prime factors are complete, therefore $\frac{6^{13}-1}{5}$ is primitive. 

For $n=14$, $a=\frac{6^{14}-1}{5}=7^2\cdot 29\cdot 197\cdot 55987$. The number $55987=\frac{6^7-1}{5}$ is primitive, so this factor has to be removed. We checked that $\frac{a}{55987}$ is complete, therefore we do not get new primitive numbers. See also Remark \ref{rem4.10}. 

For $n=15$, $a=\frac{6^{15}-1}{5}=5\cdot 43\cdot 311\cdot 1171\cdot 1201$. The factor 5 has to be removed. We checked that $\frac{a}{5}$ is incomplete and 
$\frac{a}{5\cdot 43},\frac{a}{5\cdot 311},\frac{a}{5\cdot 1171},\frac{a}{5\cdot 1201}$ are complete. Therefore $\frac{6^{15}-1}{5\cdot 5}$ is primitive. 

For $n=16$, $a=\frac{6^{16}-1}{5}=7\cdot 17\cdot 37\cdot 1297\cdot 98801$. The number $7\cdot 37\cdot 1297=\frac{6^8-1}{5}$ is primitive, so one of these factors has to be removed. We checked that $\frac{a}{7},\frac{a}{37},\frac{a}{1297}$ are incomplete. 
Then we checked that $\frac{a}{7\cdot 17}$ is incomplete and $\frac{a}{7\cdot 37},\frac{a}{7\cdot 1297},\frac{a}{7\cdot 98801}$ are complete. This implies that $\frac{a}{7\cdot 17}=\frac{6^{16}-1}{5\cdot 7\cdot 17}$ is primitive, because we cannot drop any more factors.  
Also we checked that $\frac{a}{17\cdot 37},\frac{a}{17\cdot 1297},\frac{a}{17\cdot 98801}$ are incomplete and $\frac{a}{37\cdot 1297},\frac{a}{37\cdot 98801},\frac{a}{1297\cdot 98801}$ are complete. We see now that $\frac{a}{17\cdot 37}=\frac{6^{16}-1}{5\cdot 17\cdot 37}$, $\frac{a}{17\cdot 1297}=\frac{6^{16}-1}{5\cdot 17\cdot 1297},\frac{a}{17\cdot 98801}=\frac{6^{16}-1}{5\cdot 17\cdot 98801}$ are primitive.

We can go on like that for larger values of $n$. 

A nice example is for $n=20$. Then $a=\frac{6^{20}-1}{3}=5^2\cdot 11\cdot 17\cdot 31\cdot 41\cdot 61681$. And we discover a primitive number $5^2\cdot 41\cdot 618681$ which is not square free, thus disproving a conjecture formulated by the first author in \cite{DuHa16}. 

\end{example}

\subsection{Composite numbers}\label{sec2.6}

\begin{lemma}\label{lem2.24}
Let $a,b > 1$ be odd numbers. Assume that $o_g(ab) \geq \frac{\frac{a}{g-1}-\frac{2}{g}-1+g}{\frac{g}{2}}o_g(b)$. Then $ab$ is not primitive.
\end{lemma}

\begin{proof}
Suppose that $ab$ is primitive. Then $a, b$ are relatively prime with $g$, because otherwise $ab$ is not relatively prime with $g$, so $ab$ cannot be primitive, by Lemma \ref{lem2.8}. By Proposition \ref{prop2.23}, there exists an extreme cycle $C$ and it is equal to a coset $x_0G_{ab}$ of the multiplicative group generated by $g$ in $U(\mathbb{Z}_{ab}).$ Consider the map $h : G_{ab} \rightarrow G_b, h(x) = x(\mod b)$. Then $h$ is a homomorphism and it is onto. Let $|G_{ab}| = o_g(ab) = Mo_g(b) = M|G_b|$, so that $h$ is an $M$-to-$1$ map, where $M \geq \frac{\frac{a}{g-1}-\frac{2}{g}-1+g}{\frac{g}{2}}$. Then the map $h' : x_0G_{ab} \rightarrow (x_0(\mod b))G_b,  h'(x_0x) = (x_0x)(\mod b)$, is also an $M$-to-$1$ map.

So there are exactly $M$ elements in $x_0G_{ab}$ which are mapped into $x_0(\mod b)$. These elements can be written $x_0(\mod b) + kb(\mod ab)$ for $M$ different values of $k$, each in the set $\{0, \dots, a-1\}$. Since $b$ is complete (because $ab$ is primitive), using Proposition \ref{pr2.14}, the coset $(x_0(\mod b))G_b$ contains an element greater than $\frac{2b}{g}$, and therefore we can assume that $y_0 := x_0(\mod b) > \frac{2b}{g}.$ 

From Lemma \ref{lem2.3}, we know that the cycle points are congruent to $0$ or $-ab$ modulo $g$. So $y_0+kb \equiv 0$ or $-ab$ modulo $g$ for all $M$ values of $k$ such that $y_0 + kb$ is in the extreme cycle. Since $b$ is relatively prime with $g$ it has a multiplicative  inverse $c$ in $\mathbb{Z}_{g}^{\times}$ and we have that $k \equiv -cy_0(\mod g)$ or $c(-ab-y_0)(\mod g)$. Therefore the values of $k$ here belong to only two equivalence classes modulo $g$, so in each set $A_n:=\{gn, gn+1, \dots, gn+(g-1)\}$ there are at most two values of $k$. So there are at most two values of $k$ in $A_0$, then at most two values of $k$ in $A_1$, and so on, and we must exhaust $M$ values of $k$. If $M$ is even, then we have at most $2(\frac{M}2-1)=M-2$ values of $k$ in $A_0\cup\dots \cup A_{\frac{M}2-2}$ and there are still two values of $k$ left. Therefore, if we take the largest such $k$, $k \geq g(\frac{M}{2}-1)+1$. If $M$ is odd, then a similar argument shows that $k \geq g(\frac{M-1}{2})$. In both cases, $k \geq g(\frac{M}{2}-1)+1$. Then $$y_0+kb > \frac{2b}{g} + (g(\frac{M}{2}-1)+1)b \geq \frac{ab}{g-1},$$ and this contradicts the fact that an extreme cycle is contained in $[0, \frac{ab}{g-1}],$ by Lemma \ref{lemi1.4}. 
\end{proof}

\begin{theorem}\label{th2.26}
Let $p_1, \dots, p_r$ be distinct odd primes. For $i \in \{1, \dots, r\}$, let $j_i \geq 0$ be the largest number such that $p_{i}^{j_i}$ divides $\lcm(o_g(p_1),\dots,o_g(p_r))$. Assume that $p_{1}^{\iota_{g}(p_1)+j_1} \dots p_{r}^{\iota_{g}(p_r)+j_1}$ is complete. Then $p_{1}^{k_1} \dots p_{r}^{k_r}$ is complete for any $k_1, \dots k_r \geq 0.$ 
\end{theorem}

\begin{proof}
Suppose there are some numbers $k_1, k_2, \dots, k_r \geq 0$ such that $m = p_1^{k_1} \dots p_r^{k_r}$ is not complete. Therefore, a proper divisor of this number has to be primitive, relabeling the powers $k_i$, we can assume $m$ is primitive. The hypothesis implies that for at least one $i$, $k_i \geq \iota_g(p_i) + j_i +1$. Relabeling again, we can assume $k_1 \geq \iota_g(p_1) + j_1+1$. We have, with Proposition \ref{pr2.24}: $$o_g(p_{1}^{k_1} \dots p_{r}^{k_r}) = p_{1}^{k_1-\iota_g(p_1)-j_1}o_g(p_1^{\iota_g(p_1)+j_1}p_{2}^{k_2} \dots p_{r}^{k_r}).$$
As in Lemma \ref{lem2.24}, let $a = p_{1}^{k_1-\iota_g(p_1)-j_1}, b = p_{1}^{\iota_g(p_1)+j_1}p_{2}^{k_2} \dots p_{r}^{k_r}$. We will show that $ab$ is not primitive by showing that $a > \frac{\frac{a}{g-1}-\frac{2}{g}-1+g}{\frac{g}{2}}$ for all $g$. Also, since $k_i \geq \iota_g(p_i) + j_i +1$, let $l:= k_1-\iota_g(p_1)-j_1 \geq 1$. So, we have 
$$p_1^l > \frac{\frac{p_1^l}{g-1}-\frac{2}{g}-1+g}{\frac{g}{2}} \iff \frac{g}{2}p_1^l - \frac{p_1^l}{g-1} > g-\frac{2}{g}-1 \iff \frac{p_1^l[g(g-1)-2]}{2(g-1)} > \frac{g^2-g-2}{g}$$ 
$$\iff p_1^l > \frac{2(g-1)}{g}.$$ 
Since $p_1$ is an odd prime and $l \geq 1$, $p_1^l > 2$ so it is always true that $ p_1^l > \frac{2(g-1)}{g}.$ Thus, $o_g(ab) = ao_g(b) > \frac{\frac{a}{g-1}-\frac{2}{g}-1+g}{\frac{g}{2}}o_g(b)$, so $ab$ is not primitive, a contradiction. 
\end{proof}

\begin{lemma}\label{lem2.27}
Let $m$ be incomplete and suppose that all extreme cycles for $m$ have length $o_g(m)$. Additionally, suppose that $o_g(d) < o_g(m)$ for all proper divisors $d$ of $m$. Then $m$ is primitive.  
\end{lemma}

\begin{proof}
Suppose to the contrary that $m$ is not primitive. Then $m = nk$, where $n$ is a primitive number and $k \in \mathbb{N}$. Then, with Proposition \ref{prop2.23}, $n$ has an extreme cycle $C$ of length $o_g(n)$. So $kC$ is an extreme cycle for $m$ of length $o_g(n)$, and since $o_g(n) < o_g(m)$, this contradicts that all cycles for $m$ have length $o_g(m)$. Thus $m$ is primitive. 
\end{proof}

\begin{lemma}\label{lem2.27.5}
The number of non-trivial cycle points for an odd number $m$ not divisible by $g-1$ is less than $$\min_{\textbf{n}} \Big\{2^n \big\lceil \frac{m}{(g-1)g^n} \big\rceil \Big\}.$$ 
$\lceil x \rceil$ represents the {\it ceiling }of $x$, i.e., the smallest integer larger than or equal to $x$. 
\end{lemma}

\begin{proof}
The phrasing in the statement of the lemma, "number of non-trivial cycle points," refers to the total number of points among all non-trivial cycles. 

	We know from Lemma \ref{lemi1.4} that the cycle points are contained in the intersection of the attractor $X_L$ with $\mathbb{Z}$. Also, $X_L \subset [0, \frac{m}{g-1}]$. Therefore, 
	$$X_L \subset \bigcup_{a_0, a_1, \ldots, a_{n-1} \in \{0,m\}} \sigma_{a_{n-1}} \dots \sigma_{a_0} \big[0, \frac{m}{g-1} \big] $$
	$$ = \bigcup_{a_0, a_1, \ldots, a_{n-1} \in \{0,m\}} \Bigg[\frac{a_0+ga_1+\ldots+g^{n-1}a_{n-1}}{g^n},\frac{m}{(g-1)g^n} + \frac{a_0+ga_1+\ldots+g^{n-1}a_{n-1}}{g^n}\Bigg].$$
	
	The intervals in this union can be written as
\begin{equation}
\Bigg[\frac{m\sum_{k=0}^{n-1}l_kg^k}{g^n}, \frac{m\big(1+(g-1)\sum_{k=0}^{n-1}l_kg^k\big)}{(g-1)g^n}\Bigg]
\label{eqi2.27.6}
\end{equation}
with $l_0, \ldots l_{n-1} \in \{0,1\}$. 
	Because $m$ is not divisible by $g-1$ and $1+(g-1)\sum_{k=0}^{n-1}l_kg^k$ is prime with $g-1$, the right endpoint is never an integer.
	
	There are $2^n$ intervals at each iteration, and each one contains at most $\lceil \frac{m}{(g-1)g^n} \rceil$ integers in its interior, so we have at most $2^n\lceil \frac{m}{(g-1)g^n} \rceil$ in the union. The result follows from this.  
\end{proof}

\begin{lemma}\label{lem2.29}
Let $a,b\geq1$ be odd numbers. Assume that $o_g(ab) > 2^{\lceil \log_g {\frac{a}{g-1}} \rceil}o_g(b)$. Then $ab$ is not primitive. 
\end{lemma}

\begin{proof}
Assume that $ab$ is primitive. Take $n=\lceil \log_g {\frac{a}{g-1}} \rceil$. Then $g^n \geq \frac{a}{g-1}$, so $\frac{ab}{(g-1)g^n} \leq b$, so the length of the intervals in (\ref{eqi2.27.6}) is at most $b$. Since $ab$ is primitive, there is an extreme cycle $C$ which is a coset $x_0G_{ab}$, by Proposition \ref{prop2.23}. 

	Now, as in the proof of Lemma \ref{lem2.24}, define the map $h: x_0G_{ab} \rightarrow x_0G_b, x_0x \mapsto (x_0x)(\mod b)$. We saw that this is an $M$-to-1 map. Note that $M=o_g(ab)/o_g(b) > 2^n$. There are $M$ cycle points in $x_0G_{ab}=C$ which are mapped by $h$ into $x_0$, i.e., there are $M$ values of $k$ such that $x_0(\mod b) + kb$ is in the cycle $C$. However, the intervals in (\ref{eqi2.27.6}) contain at most one such cycle point, since their length is less than $b$ and the difference between any two such points is at least $b$. We have $2^n < M$ such intervals, and this leads to a contradiction. 
\end{proof}

\begin{theorem}\label{th2.30}
Let $m$ be an odd number. Assume the following conditions are satisfied:
\begin{enumerate}
	\item For every proper divisors $d|m, d<m$, the number $d$ is complete.
	\item The following inequality holds: $$o_g(m) > \min_{\textbf{n}} \Big\{2^n \big\lceil \frac{m}{(g-1)g^n} \big\rceil \Big\}.$$
\end{enumerate}
Then $m$ is complete. If only condition (ii) is satisfied, then $m$ is not primitive. 
\end{theorem}

\begin{proof}
	Suppose (i) and (ii) hold. Then $m$ is either complete or primitive. If $m$ is primitive, then by  Proposition \ref{prop2.23} there exists a cycle of length $o_g(m)$. Since $o_g(m) > \min_{n\in \bn} \Big\{2^n \big\lceil \frac{m}{(g-1)g^n} \big\rceil \Big\}$, this contradicts Lemma \ref{lem2.27.5}. Thus $m$ is complete. 
	
	Suppose only (ii) holds. By the same argument, $m$ is not primitive. 
\end{proof} 

\begin{corollary}\label{cor2.32}
Let $m$ be an odd number. If $$o_g(m)>2^{\lceil \log_g {\frac{m}{g-1}} \rceil}$$
or in particular, if $$o_g(m) > 2\left({\frac{m}{g-1}}\right)^{\frac{1}{\log_2 g}}$$ then $m$ is not primitive. 
\end{corollary}

\begin{proof}
Let $n = \lceil \log_g {\frac{m}{g-1}} \rceil$. Then $g^n \geq \frac{m}{g-1}$ so $\lceil \frac{m}{(g-1)g^n} \rceil = 1.$ Furthermore, 
$$2^n \lceil \frac{m}{(g-1)g^n} \rceil = 2^n \leq 2^{\log_g{\frac{m}{g-1}} + 1} = 2\left({\frac{m}{g-1}}\right)^{\frac{1}{\log_2 g}}. 
$$
The rest follows from Theorem \ref{th2.30}. 
\end{proof}

\begin{corollary}\label{cor2.34}
Let $p_1, \ldots, p_r$ be distinct simple prime numbers strictly larger than $g-1$. Assume the following conditions are satisfied: 
\begin{enumerate}
	\item For any proper subset $F \subset \{1, \ldots, r\}$ and any powers $k_i \geq 0, i \in F$, the number $\prod_{i\in F} p_i^{k_i}$ is complete.
	\item None of the numbers $o_g(p_1), \dots ,o_g(p_r)$ is divisible by any of the numbers $p_1, \ldots, p_r$. 
	\item The following equation is satisfied:
	\begin{equation}
	\lcm(o_g(p_1), \ldots, o_g(p_r)) > 2^{\lceil \log_{g} \frac{p_1\dots p_r}{g-1} \rceil}
	\label{eq2.5}
	\end{equation}
\end{enumerate}
Then $p_1^{k_1} \dots p_r^{k_r}$ is complete. 
\end{corollary}

\begin{proof}
Suppose there exists $k_1, \ldots, k_r$ such that $p_1^{k_1}\dots p_r^{k_r}$ is not complete. Then pick $k_1, \ldots, k_r$ such that $\sum_{i=1}^{r} k_i$ is as small as possible, with this property. Clearly, by (i) we can assume all $k_i \geq 1$. Then all proper divisors of $p_1^{k_1}\dots p_r^{k_r}$ are complete, because otherwise we could have picked smaller $\sum k_i$. So $m := p_1^{k_1} \dots p_r^{k_r}$ is primitive. By Propositions \ref{prop2.17} and \ref{prop2.19}, we have
$$o_g(m) = \lcm(o_g(p_1^{k_1}), \ldots, o_g(p_r^{k_r})) = \lcm(p_1^{k_1-1}o_g(p_1), \ldots, p_r^{k_r-1}o_g(p_r))$$
$$ = p_1^{k_1-1} \dots p_r^{k_r-1}\lcm(o_g(p_1),\ldots,o_g(p_r)).$$

From (iii), we get
$$p_1^{k_1-1} \dots p_r^{k_r-1}\lcm(o_g(p_1),\ldots,o_g(p_r)) > 2^{\lceil \log_{g} \frac{p_1\dots p_r}{g-1} \rceil} p_1^{k_1-1} \dots p_r^{k_r-1}. $$

As in Corollary \ref{cor2.32}, letting $n = \lceil \log_g \frac{p_1\dots p_r}{g-1} \rceil$, we have $g^n \geq \frac{p_1 \dots p_r}{g-1}$ and $\lceil \frac{p_1 \dots p_r}{(g-1)g^n} \rceil = 1$. Then
$$ 2^{\lceil \log_{g} \frac{p_1\dots p_r}{g-1} \rceil} p_1^{k_1-1} \dots p_r^{k_r-1} = 2^n p_1^{k_1-1} \dots p_r^{k_r-1} = 2^n \lceil \frac{p_1 \dots p_r}{(g-1)g^n} \rceil p_1^{k_1-1} \dots p_r^{k_r-1}  \geq 2^n \lceil \frac{p_1^{k_1} \dots p_r^{k_r}}{(g-1)g^n}\rceil.$$
We used the fact that, for $a>0$ and $N\in\bn$, $\lceil a\rceil N$ is an integer bigger than $aN$, so it is bigger than $\lceil aN\rceil$. 

Thus, we obtain that 
$$o_g(m)>2^n \lceil \frac{p_1^{k_1} \dots p_r^{k_r}}{(g-1)g^n}\rceil.$$

Since $m$ is primitive, this is a contradiction to Theorem \ref{th2.30}. 
\end{proof}

\begin{corollary}\label{cor2.35}
Let $g$ be a perfect square. Let $p_1, \ldots, p_r$ be distinct simple prime numbers strictly larger than $g-1$. Assume the following conditions are satisfied: 
\begin{enumerate}
\item None of the numbers $o_g(p_1), \ldots, o_g(p_r)$ is divisible by any of the numbers $p_1, \ldots, p_r$. 
\item For any subset $\{i_1, \ldots, i_s\}$ of $\{1, \ldots, r\}$, with $s\geq2$ the following inequality holds:
	\begin{equation}
	\lcm(o_g(p_{i_1}), \ldots, o_g(p_{i_s})) > \frac{2}{(g-1)^{\frac{1}{\log_2g}}}\left({p_{i_1}\dots p_{i_m}}\right)^{\frac{1}{\log_2 g}}
	\label{eq2.6}
	\end{equation}
\end{enumerate}
Then the number $p_1^{k_1} \dots p_r^{k_r}$ is complete for any $k_1 \geq 0, \ldots , k_r \geq 0$. 
\end{corollary}

\begin{proof}
We proceed by induction on $r$. Theorem \ref{th2.18} shows that we have the result for $r=1$. Assume the result holds for $r-1$ primes. Then conditions (i) and (ii) in Corollary \ref{cor2.34} are satisfied. We check condition (iii). Let $m:= p_1\dots p_r$. 

We have, using Proposition \ref{prop2.17} in the last equality: 
\begin{equation}
2^{\lceil \log_g{\frac{m}{g-1}}\rceil} \leq 2^{\log_g{\frac{m}{g-1}} + 1} = 2\left({\frac{m}{g-1}}\right)^{\frac{1}{\log_2 g}} < o_g(m)=\text{lcm}(o_g(p_1),\dots,o_g(p_k))
\label{eq2.7}
\end{equation} 
Thus condition (iii) is satisfied and Corollary \ref{cor2.34} gives us the result. 
\end{proof}

\begin{remark}
From Theorem \ref{th2.18} we also have that $p^n$ is complete whenever $o_g(p^n)$ is even. However, as we saw in Remark \ref{rem2.14}, there are some primes which are not complete, so condition (i) in Corollary \ref{cor2.34} is not satisfied in general for an arbitrarily chosen $g$. This is why we chose $g$ to be a perfect square. 
\end{remark}

\begin{corollary}
Let $g$ be a perfect square. Let $p_1, \ldots, p_r$ be distinct simple prime numbers strictly larger than $g-1$. Assume the following conditions are satisfied: 
\begin{enumerate}
\item The numbers $o_g(p_1), \ldots, o_g(p_r), p_1, \ldots, p_r$ are mutually prime. 
\item The following inequality holds $$o_g(p_j) > \sqrt{\frac{2}{(g-1)^{\frac{1}{\log_2g}}}}\cdot p_j^{\frac{1}{\log_2 g}}$$ for all $j$
\end{enumerate}
\end{corollary}

\begin{proof}
Note first that $2>(g-1)^{\frac{1}{\log_2 g}}$. We use Corollary \ref{cor2.35}. For any subset $\{i_1, \ldots, i_s\}$ of $\{1, \ldots, r\}$ with $s \geq 2$, we have
$$
\text{lcm}(o_g(p_{i_1}),\dots, o_g(p_{i_s}))=o_g(p_{i_1})\dots o_g(p_{i_s}) \geq \left(\sqrt{\frac{2}{(g-1)^{\frac{1}{\log_2 g}}}}\right)^s\left(p_{i_1}\dots p_{i_s}\right)^{\frac{1}{\log_2g}}$$$$\geq \frac{2}{(g-1)^{\frac{1}{\log_2g}}}\left({p_{i_1} \dots p_{i_s}}\right)^\frac{1}{\log_2 g}.$$ 
\end{proof}

\begin{corollary}\label{cor2.37.1}
Let $a$ be a complete odd number. Let $p > g-1$ be a simple prime number. Assume that
\begin{enumerate}
\item p does not divide a
\item $o_g(p)$ and $o_g(a)$ are mutually prime
\item $o_g(p) > 2^{\lceil \log_g \frac{p}{g-1} \rceil}$
\end{enumerate}
Then $p^ka$ is complete for all $k \geq 0$. 
\end{corollary}

\begin{proof}
Since $p$ does not divide $a$, $p^k$ is mutually prime with $a$. Since $p$ is simple, with Propositions \ref{prop2.17} and \ref{prop2.19} we have 
$$o_g(p^ka) = \lcm(o_g(p^k),o_g(a)) = p^{k-1}o_g(p)o_g(a). $$ 
So then $p^{k-1}o_g(p) > p^{k-1}2^{\lceil \log_g \frac{p}{g-1} \rceil}$, by hypothesis. Taking the $\log_2$ of both sides, with $k\geq2$ and $p \geq g+1$ we get
$$ \log_2 (p^{k-1}o_g(p)) > \log_2 (p^{k-1}2^{\lceil \log_g \frac{p}{g-1} \rceil}) = \log_2 p^{k-1} + {\lceil \log_g \frac{p}{g-1} \rceil} $$
$$ \geq \log_g p^{k-1} + 1 + {\lceil \log_g \frac{p}{g-1} \rceil} \geq \lceil{\log_g p^{k-1}\rceil} + {\lceil \log_g \frac{p}{g-1} \rceil} \geq\lceil{\log_g \frac{p^{k}}{g-1}\rceil}.$$

We used here the fact that $$\log_2p^{k-1}\geq \log_gp^{k-1}+1\Longleftrightarrow p^{k-1}g\leq (p^{k-1})^{\frac1{\log_g2}},$$
which is true, because $\log_g2\leq \frac12$, since $g\geq 4$ and $p>g$.

Therefore 
$$p^{k-1}o_g(p) > 2^{\lceil{\log_g \frac{p^{k}}{g-1}\rceil}},$$
for $k \geq 2$ and also for $k=1$ by hypothesis. By Lemma \ref{lem2.29}, $p^ka$ cannot be primitive, for $k \geq 1$ and because $a$ is complete and $p$ is prime, this means that $p^ka$ is complete. 
\end{proof}

\begin{example}
Let $g = 16$. We want to prove that $17^k \cdot 19^l$ is complete for any $k,l$. We have $o_{16}(17) = 2, o_{16}(17^2) = 34, o_{16}(19) = 9,$ and $o_{16}(19^2) = 171$, so $17$ and $19$ are both simple primes. Since $g$ is a perfect square, by Theorem \ref{th2.18} , $17^k$ and $19^l$ are complete for any $k,l$. Also, $$\lcm(o_{16}(17), o_{16}(19)) = 18 > 2^{\lceil \log_{16}\frac{17\cdot 19}{15}\rceil} = 4.$$ The result follows from Corollary \ref{cor2.34}. 
\end{example}

\begin{example}
Let $g=36$. We want to prove that $37^k \cdot 43^l$ is complete for any $k, l$. Since $g$ is a perfect square, $37^k$ is complete for any $k$ by Theorem \ref{th2.18}. Also, $o_{36}(43) = 3, o_{36}(43^2) = 129$, so $43$ is a simple prime, and $o_{36}(37) = 2$ is mutually prime with $o_{36}(43) = 3$. In addition, $$o_{36}(43) = 3 > 2^{\lceil \log_{36}\frac{43}{35}\rceil} = 2,$$ so the result follows from Corollary \ref{cor2.37.1}. 

The same argument applies to show that $37^k \cdot 47^l, 37^k \cdot 53^l, 37^k \cdot 59^l, 37^k \cdot 67^l, 37^k \cdot 71^l$ are complete. We can use this argument also for $47^k \cdot 53^l \cdot 59^j$. First, note that $47, 53$ and $59$ are simple primes with $o_{36}(47) = 23, o_{36}(53) = 13,$ and $o_{36}(59) = 29$. Then $47^k \cdot 53^l$ is complete by Corollary \ref{cor2.37.1}. By Propositions \ref{prop2.17} and \ref{prop2.19}, $o_{36}(47^l\cdot 53^k)$ is relatively prime with $o_{36}(59)$, so $47^k \cdot 53^l \cdot 59^j$ is also complete by Corollary \ref{cor2.37.1}. 
\end{example}

\begin{example}
Let $g$ be any even perfect square less than $1000$. We will show $907^k\cdot 911^l$ is complete for any $k,l$. With a computer check, $907, 911$ are both simple primes for every even perfect square less than $1000$. Moreover, $o_g(907)$ and $o_g(911)$ are relatively prime for each $g$. With another computer check, we also have that 
$$o_g(907) > 2^{\lceil \log_{g} \frac{907}{g-1}\rceil} \text{ and } o_g(911) > 2^{\lceil \log_{g} \frac{911}{g-1}\rceil}$$ for all $g$, so $907^k\cdot 911^l$ is complete by Corollary \ref{cor2.37.1}. 
\end{example}

 \noindent {\it Acknowledgments} This material is based upon work supported by the National Science Foundation under Award No. 1356233. This work was partially supported by a grant from the Simons Foundation (\#228539 to Dorin Dutkay) 

\bibliographystyle{alpha}	
\bibliography{eframes}

\begin{thebibliography}{JKS14b}

\bibitem[DH16]{DuHa16}
Dorin~Ervin Dutkay and John Haussermann.
\newblock Number theory problems from the harmonic analysis of a fractal.
\newblock {\em J. Number Theory}, 159:7--26, 2016.

\bibitem[DHL13]{MR3055992}
Xin-Rong Dai, Xing-Gang He, and Chun-Kit Lai.
\newblock Spectral property of {C}antor measures with consecutive digits.
\newblock {\em Adv. Math.}, 242:187--208, 2013.

\bibitem[DJ06]{DJ06}
Dorin~Ervin Dutkay and Palle E.~T. Jorgensen.
\newblock Iterated function systems, {R}uelle operators, and invariant
  projective measures.
\newblock {\em Math. Comp.}, 75(256):1931--1970 (electronic), 2006.

\bibitem[DJ07]{DJ07d}
Dorin~Ervin Dutkay and Palle E.~T. Jorgensen.
\newblock Fourier frequencies in affine iterated function systems.
\newblock {\em J. Funct. Anal.}, 247(1):110--137, 2007.

\bibitem[DJ12]{DJ12}
Dorin~Ervin Dutkay and Palle E.~T. Jorgensen.
\newblock Fourier duality for fractal measures with affine scales.
\newblock {\em Math. Comp.}, 81(280):2253--2273, 2012.

\bibitem[Hut81]{Hut81}
John~E. Hutchinson.
\newblock Fractals and self-similarity.
\newblock {\em Indiana Univ. Math. J.}, 30(5):713--747, 1981.

\bibitem[JKS12]{MR2966145}
Palle E.~T. Jorgensen, Keri~A. Kornelson, and Karen~L. Shuman.
\newblock An operator-fractal.
\newblock {\em Numer. Funct. Anal. Optim.}, 33(7-9):1070--1094, 2012.

\bibitem[JKS14a]{MR3202868}
Palle E.~T. Jorgensen, Keri~A. Kornelson, and Karen~L. Shuman.
\newblock Scalar spectral measures associated with an operator-fractal.
\newblock {\em J. Math. Phys.}, 55(2):022103, 23, 2014.

\bibitem[JKS14b]{MR3310953}
Palle E.~T. Jorgensen, Keri~A. Kornelson, and Karen~L. Shuman.
\newblock Scaling by 5 on a {$\frac{1}{4}$}-{C}antor measure.
\newblock {\em Rocky Mountain J. Math.}, 44(6):1881--1901, 2014.

\bibitem[JP98]{JP98}
Palle E.~T. Jorgensen and Steen Pedersen.
\newblock Dense analytic subspaces in fractal {$L\sp 2$}-spaces.
\newblock {\em J. Anal. Math.}, 75:185--228, 1998.

\bibitem[Li07]{MR2297038}
Jian-Lin Li.
\newblock {$\mu\sb {M,D}$}-orthogonality and compatible pair.
\newblock {\em J. Funct. Anal.}, 244(2):628--638, 2007.

\bibitem[{\L}W02]{MR1929508}
Izabella {\L}aba and Yang Wang.
\newblock On spectral {C}antor measures.
\newblock {\em J. Funct. Anal.}, 193(2):409--420, 2002.

\bibitem[Str00]{MR1785282}
Robert~S. Strichartz.
\newblock Mock {F}ourier series and transforms associated with certain {C}antor
  measures.
\newblock {\em J. Anal. Math.}, 81:209--238, 2000.

\end{thebibliography}

\end{document}